\documentclass[10pt, reqno]{amsart}

\usepackage[utf8]{inputenc}
\usepackage{amsmath}
\usepackage{amssymb}
\usepackage{mathrsfs}
\usepackage{graphics}
\usepackage{geometry}
\usepackage{a4wide}
\usepackage{setspace}
\usepackage{color}
\usepackage{amsthm}
\usepackage{paralist}
\usepackage{dsfont}
\usepackage{verbatim}
\usepackage[numbers]{natbib}

\newcommand{\dd}{\,\mathrm{d}}
\newcommand{\dn}{\mathrm{d}}

\newcommand{\dff}{\mathrm{D}}

\newcommand{\ent}{\mathcal{E}}
\newcommand{\krnl}{\mathbf{G}}

\newcommand{\ball}{\mathbb{B}}
\newcommand{\calH}{{\mathcal{H}}}

\newcommand{\eps}{\varepsilon}
\newcommand{\W}{\mathbf{W}}
\newcommand{\bd}{\mathbf{d}}
\newcommand{\M}{\mathbf{M}}
\newcommand{\mob}{\mathbf{m}}
\newcommand{\K}{\mathbf{K}}
\newcommand{\X}{\mathbf{X}}
\newcommand{\E}{\mathbf{E}}
\newcommand{\F}{\mathbf{F}}
\newcommand{\bH}{\mathbf{H}}
\newcommand{\bL}{\mathbf{L}}

\newcommand{\N}{{\mathbb{N}}}
\newcommand{\R}{{\mathbb{R}}}
\newcommand{\flow}{\mathsf{S}}
\newcommand{\eins}[1]{\mathbf{1}_{#1}}
\newcommand{\Id}{\mathds{1}}
\newcommand{\scrC}{\mathscr{C}}
\newcommand{\act}{\mathbf{\Phi}}

\newcommand{\dual}[2]{\left\langle #1,#2\right\rangle}
\newcommand{\Matn}{\R^{n\times n}}
\newcommand{\tT }{\mathrm{T}}
\newcommand{\tang}{\mathsf{T}}
\newcommand{\measm}{\mathscr{M}(\R;S)}
\newcommand{\meas}{\mathscr{M}(\R;\R^n)}
\newcommand{\measo}{\mathscr{M}(\R;A)}
\newcommand{\tr}{\mathrm{tr}}
\newcommand{\inn}[1]{\mathrm{int}(#1)}

\newcommand{\supp}{\mathrm{supp}\,}
\newcommand{\hess}{\nabla^2_z}
\newcommand{\grd}{\nabla_z}
\newcommand{\einsvec}{\mathsf{e}}
\newcommand{\id}{\operatorname{id}}
\newcommand{\mom}[1]{\boldsymbol{\ell}_2(#1)}
\newcommand{\pot}{\mathcal{V}}
\newcommand{\zref}{{\bar{z}}}
\newcommand{\Xaux}{\mathbf{X}_{\zref}}
\newcommand{\Yent}[3]{\ent_{#1}\left(#2\,|\,#3\right)}
\newcommand{\gau}[1]{\left\lfloor#1\right\rfloor}
\newcommand{\ban}{\mathbf{Y}}
\newcommand{\dom}{\mathrm{Dom}}

\newtheorem{definition}{Definition}[section]
\newtheorem{remark}[definition]{Remark}
\newtheorem{lemma}[definition]{Lemma}
\newtheorem{thm}[definition]{Theorem}
\newtheorem{prop}[definition]{Proposition}

\newtheorem{xmp}[definition]{Example}

\begin{document}

\begin{abstract}
  We introduce Wasserstein-like dynamical transport distances between vector-valued densities on $\R$. 
  The mobility function from the scalar theory is replaced by a mobility matrix, that is subject to positivity and concavity conditions.
  Our primary motivation is to cast certain systems of nonlinear parabolic evolution equations into the variational framework of gradient flows.
  In the first part of the paper, we investigate the structural properties of the new class of distances like geodesic completeness.
  The second part is devoted to the identification of $\lambda$-geodesically convex functionals and their $\lambda$-contractive gradient flows.
  One of our results is a generalized McCann condition for geodesic convexity of the internal energy. 
  In the third part, the existence of weak solutions to a certain class of degenerate drift-diffusion systems is shown. 
  Even if the underlying energy function is \emph{not} geodesically convex w.r.t. our new distance, 
  the construction of a weak solution is still possible using de Giorgi's \emph{minimizing movement} scheme.
\end{abstract}

\title[Transport distances and geodesic convexity]{Transport distances and geodesic convexity for systems of degenerate diffusion equations}
\author[Jonathan Zinsl]{Jonathan Zinsl}
\author[Daniel Matthes]{Daniel Matthes}
\address{Zentrum f\"ur Mathematik \\ Technische Universit\"at M\"unchen \\ 85747 Garching, Germany}
\email{zinsl@ma.tum.de}
\email{matthes@ma.tum.de}
\thanks{This research was supported by the German Research Foundation (DFG), Collaborative Research Center SFB-TR 109.}
\date{\today}
\maketitle


\section{Introduction}\label{sec:intro}

\subsection{The evolution system and its variational structure}\label{subsec:evolution}

This paper is concerned with the variational structure of the following system of coupled nonlinear evolution equations in one spatial dimension
\begin{align}
  \label{eq:pdesystem}
  \partial_t \mu(t,x)&=\partial_x\left[\M(\mu(t,x))\partial_x\ent'(\mu(t,x))\right],\quad t>0,\,x\in\R,
\end{align}
for the $n$ components $\mu_1,\ldots,\mu_n$ of the sought-for function $\mu:[0,\infty)\times\R\to S$. 
We assume that $\mu$ attains values in a convex compact set $S\subset\R^n$ with nonempty interior $\inn{S}$. 
Above, $\M:\,S\to\Matn$ is the \emph{mobility matrix},
and $\ent'$ is the first variation of the \emph{driving entropy functional} $\ent$ which is defined on $\measm$, the space of measurable functions on $\R$ with values in $S$.

Formally, \eqref{eq:pdesystem} is a gradient flow:
solutions $\mu(t,\cdot)$ are curves of steepest descent in the potential landscape of $\ent$,
with respect to the Riemannian structure induced on the ``manifold'' $\measm$
by weighted $H^{-1}$-norms $\|\cdot\|_\mu$ on ``tangent vectors'' $\dot\mu$:
\begin{align}
  \label{eq:metricformal}
  \|\dot\mu\|_\mu^2 = \int_{\R} \partial_x\Psi^\tT\M(\mu)\partial_x\Psi \dd x , 
\end{align}
where the auxiliary function $\Psi:\R\to\R^n$ solves the elliptic problem
\begin{align*}
  \dot\mu + \partial_x\big(\M(\mu)\partial_x\Psi\big) = 0.  
\end{align*}


This kind of gradient flow structure is well-known in the scalar case $n=1$, 
where the mobility matrix $\M$ simplifies to a scalar mobility function $\mob$.
In particular, if $\mob(z)=z$ is linear, then the metric described above is the $L^2$-Wasserstein distance.
In the last decade, quite a few well-known evolution equations have been identified as gradient flows in the Wasserstein metric,
and have been rigorously analysed on grounds of that special property, 
among them
the Fokker-Planck \cite{jko1998}, 
the porous media \cite{otto2001}, 
the nonlocal aggregation \cite{cmv2006},
the Hele-Shaw \cite{giacotto2001} 
and the fourth order quantum \cite{gianazza2009} equations.
Concerning gradient flows in metrics defined by nonlinear mobility functions $\mob$,
we refer to \cite{carrillo2010,lisini2012}.

Comparatively little interest has been devoted 
to the gradient flow structure of genuine systems \eqref{eq:pdesystem} with $n>1$ components. Systems of that kind arise e.g. in reaction-diffusion models for chemical agents as well as for semiconductor dynamics \cite{mielke2011, liero2012, mielke2013, glitzky2013}, or for population dynamics \cite{chen2006, juengel2012, juengel2014, burger2013} with or without cross-diffusion. 

Our main motivation is to analyse systems of the form \eqref{eq:pdesystem} by means of their variational structure. So far, there does not seem to exist any rigorous study of the properties of the metric induced by \eqref{eq:metricformal}.
In this paper, we provide an extension of several aspects of the rich scalar theory to vector-valued densities $\mu$.
Specifically:
\begin{itemize}
\item we define a metric $\W_\M$ on $\measm$ that gives a rigorous interpretation to the formal structure \eqref{eq:metricformal} above,
  we study the topological properties of the space $(\measm,\W_\M)$,
  and establish connections to the $L^2$-Wasserstein distance;
\item we derive sufficient conditions for geodesic $\lambda$-convexity with respect to $\W_\M$ for a class of entropy functionals $\ent$;
\item we prove existence of weak solutions for a certain class of degenerate drift-diffusion equations of the form \eqref{eq:pdesystem},
  even without convexity hypotheses on the respective $\ent$.
\end{itemize}
We proceed with a summary of our results and relate them to the existing literature.

\subsection{Study of the new metric}
A cornerstone in the theory of optimal transportation is the Benamou-Brenier dynamical interpretation 
of the $L^2$-Wasserstein distance \cite{brenier2000}.
Dolbeault \emph{et al.} \cite{dns2009} have used that interpretation to \emph{define} a new class of transportation metrics $\W_\mob$, 
corresponding to \emph{nonlinear} mobilities $\mob$.
Their results have been generalized by Lisini and Marigonda \cite{lisini2010},
and properties of the new metrics have been investigated in \cite{carrillo2010,carda2013}.
In a nutshell, the Benamou-Brenier formula leads to well-defined metrics $\W_\mob$
if the mobility function $\mob:S\to\R$ is positive and concave on the interior of the (possibly semi-infinite) interval $S\subset \R$.

We extend the approach of \cite{dns2009, lisini2010} to densities $\mu:\R\to S$ with values in a convex and compact set $S\subset\R^n$, and a mobility matrix $\M:S\to\Matn$ in place of $\mob$.
Our hypotheses on $\M$ are:
\begin{enumerate}[(C1)]
\setcounter{enumi}{-1}
\item $\M:\,S\to\Matn$ is continuous, and is smooth on $\inn{S}$.
\item $\M(z)$ is symmetric and positive definite for each $z\in\inn{S}$.
\item $\dff^2\M(z)[\zeta,\zeta]$ is negative semidefinite for each $z\in \inn{S}$ and $\zeta\in\R^n$.
\item $\M(z)\nu=0$ if $z\in\partial S$ and $\nu$ is a normal vector to $\partial S$ at $z$.
\end{enumerate}
Conditions (C1)--(C2) are direct generalizations of positivity and concavity of the mobility function $\mob$,
and (C0) is a technical hypothesis.
Condition (C3) is a natural requirement that is satisfied in all of our examples, but is not substantial for the proofs.
Its intepretation is that the values of solutions to \eqref{eq:pdesystem} are confined to $S$.

Finally, whenever discussing specific examples, we shall further assume 
that $\M$ is \emph{induced by a function} $h\in C^2(\inn{S})$, 
which means that
\begin{align}
  \label{eq:Minduce}
  \M(z) = (\hess h(z))^{-1} \quad \text{for all $z\in\inn{S}$}.
\end{align}
This hypothesis allows to formulate the multi-component heat equation $\partial_t\mu=\partial_{xx}\mu$ in the form \eqref{eq:pdesystem}; 
with the functional $\ent(\mu)=\int_\R h(\mu(x))\dd x$.

Under conditions (C0)--(C3), we prove that the Benamou-Brenier formula with the norms from \eqref{eq:metricformal} 
defines a (pseudo-) metric $\W_\M$ on the space $\measm$. 
Moreover, by a careful transfer of the proofs in \cite{dns2009, lisini2010} to the multi-component setting,
we obtain that $\W_\M$ inherits the essential topological properties known for the $\W_\mob$ distances,
like
\begin{itemize}
\item existence of constant-speed geodesics connecting densities of finite distance,
\item lower semicontinuity with respect to weak$\ast$ convergence,
\item weak$\ast$-relative compactness of bounded sets.
\end{itemize}
We further discuss under which criteria $\W_\M$ is finite.

In practice, the conditions (C0)--(C3) turn out to be quite restrictive, 
and their validity is fragile under perturbations.
A seemingly trivial family of examples is given by the \emph{fully decoupled} mobility matrices,
\begin{align}
  \label{eq:fullydecoupled}
  \M(z) =
  \begin{pmatrix}
    \mob_1(z_1) & & & \\ & \mob_2(z_2) & & \\ & & \ddots & \\ & & & \mob_n(z_n)
  \end{pmatrix},
\end{align}
with $n$ non-negative, concave (scalar) mobility functions $\mob_k:[a_k,b_k]\to\R$.
Since the components do not interact with each other through $\M$,
one has that $\W_\M^2=\W_{\mob_1}^2+\cdots+\W_{\mob_n}^2$, 
i.e., $\W_\M$ is simply the sum of the metrics for the components $\W_{\mob_k}$.
Somewhat surprisingly, it turns out that fully decoupled matrices are ungeneric for property (C2) in the sense
that any sufficiently general, arbitrarily small perturbation of the components of $\M$ makes (C2) invalid.
We shall show how certain fully decoupled matrices can be ``stabilized'' with a suitably chosen special perturbation
such that the perturbed mobility matrix retains (C2).

\subsection{Geodesic convexity}
Geodesic $\lambda$-convexity plays a pivotal role in the theory of metric gradient flows \cite{savare2008}.
For the gradient flows of $\lambda$-convex functionals, one obtains immediately
contractivity estimates, asymptotics for the long-time behaviour, and error bounds on the time-discrete implicit Euler approximation.
However, $\lambda$-convexity in transportation metrics is a very rare property \cite{mccann1997}.

Up to now, 
the only known $\lambda$-convex functionals $\ent$ for the metrics $\W_\mob$ with nonlinear mobilities $\mob$ 
in space dimension $d=1$
are the \emph{internal energies} 
\begin{align}
  \label{eq:pme}
  \ent(\mu)=\int_{\R} f\big(\mu(x)\big)\dd x,
\end{align}
provided that $f$ satisfies the generalized McCann condition \cite{carrillo2010} in $d=1$,
i.e., $s\mapsto \mob(s)^2 f''(s)$ is a nonnegative function,
and the \emph{regularized potential energies}
\begin{align}
  \label{eq:pot}
  \pot(\mu)=\int_{\R}\left[\alpha h\big(\mu(x)\big)+\rho(x)\mu(x)\right]\dd x,
\end{align} 
where $h:[0,\infty)\to\R$ is such that $\mob h''\equiv1$, $\alpha>0$,
and $\rho:\R\to\R$ is a smooth function of compact support \cite{lisini2012}.
The respective gradient flows are given by
\begin{align*}
  \partial_t\mu = \partial_{xx} P(\mu), \quad\text{and}\quad \partial_t\mu = \alpha\partial_{xx}\mu + \partial_x\big(\mob(\mu)\partial_x\rho\big),
\end{align*}
where $P'(z) = \mob(z)f''(z)$.

Both types of functionals \eqref{eq:pme} and \eqref{eq:pot} possess canonical generalizations to densities with multiple components.
In \eqref{eq:pme}, simply replace $f$ by a smooth function $f:S\to\R$.
To make sense of \eqref{eq:pot}, assume that $\M$ is induced by $h:S\to\R$, see \eqref{eq:Minduce},
and use a potential $\rho:\R\to\R^n$ with $n$ components.
We derive sufficient criteria for the geodesic $\lambda$-convexity of these functionals with respect to the new metric $\W_\M$.
For that, we use the formalism developed by Liero and Mielke \cite{mielke2011,liero2012}, 
which is based on the \emph{Eulerian calculus} for transportation distances, see \cite{daneri2008,otto2005}.

Our own generalization of McCann's condition for $\ent$ of the form \eqref{eq:pme} 
is given in Proposition \ref{prop:mccann}, see formula \eqref{eq:genmccann}.
Our examples for pairs of a (nondiagonal) mobility matrix $\M$ and a function $f$ that satisfy this condition
are currently limited to perturbations of the heat equation.
For definiteness, assume that $\M$ is induced by $h$
and choose $f=h+\eps g$, where $g:S\to\R$ vanishes near the boundary of $S$.
If $\eps$ is sufficiently small, then our generalized McCann condition holds.
The evolution equation \eqref{eq:pdesystem} specializes to a perturbation of the multi-component heat equation:
\begin{align*}
  \partial_t\mu = \partial_{xx} \mu + \eps\partial_x\big(\M(\mu)\partial_x (\grd g(\mu))\big).
\end{align*}
In contrast, if $\M$ is a fully decoupled mobility, $\ent$ has to be decoupled in order to satisfly our generalized McCann condition.

Our condition assuring $\lambda$-convexity for functionals of type \eqref{eq:pot} is given in \eqref{eq:genconv}.
Even for smooth $\rho$ of compact support, 
it imposes an apparently very strong restriction on the function $h$ in $\M(z)=(\hess h(z))^{-1}$.

\subsection{Construction of solutions to \eqref{eq:pdesystem}}
In Section \ref{sec:weak}, we discuss the primary application of the new metric $\W_\M$,
namely the construction of weak solutions to a class of drift-diffusion equations of the form \eqref{eq:pdesystem}
by means of de Giorgi's \emph{minimizing movement scheme},
which is a time-discrete implicit Euler scheme for gradient flows (cf. \cite{jko1998}, see also Section \ref{sec:pre} below).
Specifically, we consider the initial value problem
\begin{align}
  \label{eq:intropde}
  \partial_t \mu=\partial_x(\M(\mu)\hess f(\mu)\partial_x\mu+\M(\mu)\partial_x\eta),
  \quad \mu(0)=\mu^0,
\end{align}
where the mobility matrix $\M$ is fully decoupled as in \eqref{eq:fullydecoupled}, $S\subset \R^n$ is an $n$-cuboid 
and $f:S\to\R$ is uniformly convex, $\hess f(z)\ge C_f\Id$ with $C_f>0$, 
but does \emph{not} need to be the sum of functions of the components of $\mu$.
Thus, the diffusion matrix $\M\hess f$ will not be symmetric nor positive definite in general.
Also, the corresponding energy functional
\begin{align*}
  \ent(\mu) = \int_\R f\big(\mu(x)\big)\dd x + \int_\R \mu(x)^\tT\eta(x)\dd x
\end{align*}
will not be $\lambda$-convex.

Still, the variational minimizing movement scheme is well-posed.
We prove that in the limit of vanishing time step size, 
it produces a limit curve that is a weak solution to \eqref{eq:intropde}, see Theorem \ref{thm:exist}.
The crucial \emph{a priori} estimate for the passage to the limit is provided 
by the dissipation of the auxiliary functional $\int_\R h(\mu(x))\dd x$, 
where $h$ induces $\M$ according to \eqref{eq:Minduce}.
That dissipation amounts to
\begin{align}
  \label{eq:introapriori}
  -\frac{\dd}{\dd t}\int_\R h\big(\mu\big)\dd x 
  \ge C_f\int_\R |\partial_x\mu|^2 \dd x + \int_\R\partial_x\eta^\tT\partial_x\mu\dd x,
\end{align}
and therefore provides square-integrability of $\partial_x\mu$ in space and time.

We emphasize that the global existence of solutions to \eqref{eq:intropde} is a nontrivial result of independent interest.
It does not follow immediately from classical parabolic theory:
Indeed, since $\M\hess f$ typically lacks positivity (meaning that $v^\tT\M\hess fv\ge0$),
the differential operator in \eqref{eq:intropde} is not elliptic in the strong sense.
The theory for parabolic equations with normally elliptic operators, see e.g. \cite{amann},
provides existence of solutions only locally in time for sufficiently regular initial data;
for extension of those to global solutions, additional estimates would be needed which guarantee 
that the values of the solution $\mu$ stay away from the boundary of the admissible set $S$.

Our gradient flow approach is conceptually different.
The minimizing movement scheme naturally produces a globally defined curve $\mu$ with values in $S$ in the continuous time limit.
Instead, the main step of the proof is to identify this limit curve as a weak solution to \eqref{eq:intropde},
using the compactness induced by \eqref{eq:introapriori}.
In comparison to the classical results, we obtain weaker solutions of lower regularity, 
but we can allow for more general initial data.

We remark that just recently, 
a closely related class of \emph{reaction}-diffusion systems 
has been studied \cite{juengel2014} on grounds of a very similar dissipation estimate as in \eqref{eq:introapriori},
see also \cite{burger2010}.
However, the equations considered in \cite{juengel2014} are generically not of gradient flow type, 
and a completely different technique, based on a suitable transformation of the dependent variables, 
has been employed for approximation of global weak solutions with values in a prescribed set $S$.
The connection to our treatment of \eqref{eq:intropde} is that in both cases, 
subtle structural properties of the diffusion matrix play a pivotal role.

\subsection{Plan of the paper}\label{subsec:plan}
In Section \ref{sec:pre}, we introduce basic notation and definitions.
Then we provide a couple of examples for mobility matrices $\M$ satisfying (C0)--(C3) in Section \ref{sec:mobex}. 
In Section \ref{sec:metric}, the general theory for distances generated by a mobility matrix is developed. 
There, we start with basic properties of the objects occuring in the definition of these new distances (Section \ref{subsec:def_metr}) 
before investigating solutions to the continuity equation (Section \ref{subsec:sol_cont}). 
This enables us to prove the defining properties of the distance and several additional topological properties in Section \ref{subsec:topo}. 

Section \ref{sec:geodconv} is devoted to geodesic convexity and gradient flows 
with respect to the distances established in Section \ref{sec:metric} 
and begins with an introduction of the abstract background (Section \ref{subsec:geo_pre}). 
We continue with the investigation of the multi-component heat equation (Section \ref{subsec:heat}) 
and general internal energy functionals (Section \ref{subsec:funct}), 
where specific perturbation results are also given. 
Afterwards, we study geodesic convexity of the regularized potential energy functionals in Section \ref{subsec:xfunct}.

In Section \ref{sec:weak}, we first introduce the framework (Section \ref{subsec:sett}) before constructing an approximate time-discrete solution to the given equation in Section \ref{subsec:timediscrete}. The desired weak solution is then obtained by passage to the continuous-time limit (Section \ref{subsec:conttime}).


\section{Preliminaries}\label{sec:pre}
We first introduce our notation before stating relevant definitions and statements and refer to \cite{ambrosio2000, villani2003, savare2008} for more details on optimal transportation, gradient flows, and their measure theoretic preliminaries.

\subsection{Basic notation}

Components of a $n$-vector $v\in\R^n$ are indicated with lower indices: $v=(v_1,v_2,\ldots,v_n)$. 
By $|\cdot|$, we denote the usual Euclidean norm and inner product on $\R^n$, 
whereas $\dual{\cdot}{\cdot}$ formally denotes the duality pairing on $L^2(\R;\R^n)$. 
Inequalities between vectors, multi-dimensional intervals (also referred to as $n$-cuboids) $[q_0,q_1]$ for $q_0,\,q_1\in\R^n$, $q_0\le q_1$, as well as integration of vector-valued functions are understood component-wise. 

We use $\grd$ for the gradient, $\hess$ for the Hessian 
and $\dff$ in combination with square brackets for directional derivatives with respect to $z$.
For instance, if $\M:\,S\to\Matn$ and $\mu:\,\R\to S$, 
then we write the chain rule as
\begin{align*}
  \partial_x\M(\mu) = \dff\M(\mu)[\partial_x\mu].
\end{align*}
Note that even for symmetric matrices $\M$, the third order-tensor $\dff\M$ and the fourth-order tensor $\dff^2\M$ are not totally symmetric in general, 
although $\dff \M(z)[\zeta]$ and $\dff^2 \M(z)[\zeta,\tilde\zeta]$ are symmetric $n\times n$ matrices, for every choice of $\zeta,\tilde\zeta\in\R^n$. Given a multilinear operator or its tensor representative, the norm $\|\cdot\|$ denotes the operator norm.

For a nonnegative measurable function $\tilde\mu:\,\R\to\R^n$, the functional 
\begin{align*}
\mom{\tilde\mu}:=\int_\R x^2\einsvec^\tT\tilde\mu(x)\dd x\in \R\cup\{\infty\}
\end{align*}
is called the \emph{second moment} of $\tilde\mu$, where $\einsvec:=(1,1,\ldots,1)^\tT\in\R^n$.
The space of nonnegative functions $\nu\in L^1(\R)$ with fixed mass and finite second moment can be equipped with the \emph{$L^2$-Wasserstein distance} $\W_2$:
\begin{align*}
  \W_2(\nu,\tilde\nu)&:=\inf\left[\int_{\R\times\R}|x-y|^2d\boldsymbol{\nu}:\,\boldsymbol{\nu}\in\Gamma(\nu,\tilde\nu)\right]^{1/2},
\end{align*}
where $\Gamma(\nu,\tilde\nu)$ denotes the set of all couplings between the two finite Borel measures with density $\nu$ and $\tilde\nu$, respectively. 
By \cite{brenier2000}, one has the equivalent dynamic characterization of the Wasserstein distance:
\begin{align*}
  \W_2(\nu,\tilde \nu)&=\inf\left[\int_0^1\int_\R \hat\nu(t,x)|v(t,x)|^2\dd x\dd t:\,\partial_t\hat\nu+\partial_x(\hat\nu v)=0\text{ in }(0,1)\times\R,\,\hat\nu|_{t=0}=\nu,\,\hat\nu|_{t=1}=\tilde\nu\right]^{1/2}.
\end{align*}
Given a closed set $A\subset\R^n$, $\measo$ denotes the space of all measurable functions $\tilde\mu:\,\R\to A$. We call a sequence of measurable functions $(\tilde\mu_k)_{k\in\N}$ in $\measo$ weak$\ast$-convergent to its limit $\tilde\mu\in\measo$, if for all $\rho\in C^0_c(\R;\R^n)$, one has
\begin{align*}
\lim_{k\to\infty}\int_\R \tilde\mu_k^\tT\rho\dd x=\int_\R \tilde\mu^\tT\rho\dd x.
\end{align*}

\subsection{Gradient flows in metric spaces}

For general metric spaces $(\X,\bd)$, a functional $\mathcal{A}:\,\X\to\R\cup\{\infty\}$ is called $\lambda$-geodesically convex w.r.t. $\bd$ for some $\lambda\in\R$, if for all $u_0,u_1\in\mathcal{A}$ and all $t\in[0,1]$, one has
\begin{align*}
  \mathcal{A}(u_t)&\le (1-t)\mathcal{A}(u_0)+t\mathcal{A}(u_1)-\frac{\lambda}{2}t(1-t)\bd^2(u_0,u_1),
\end{align*}
where $u_t:\,[0,1]\to\X$ is a geodesic curve connecting $u_0$ and $u_1$. We introduce the notion of \emph{$\lambda$-contractive gradient flow} by means of the following 
\begin{definition}[$\lambda$-flow]
 Let $\mathcal{A}:\,\X\to\R\cup\{\infty\}$ be a lower semicontinuous functional on the metric space $(\X,\bd)$.
 A continuous semigroup $\flow$ on $(\X,\bd)$ is called \emph{$\lambda$-flow} for some $\lambda\in\R$,
 if the \emph{evolution variational estimate} (with parameter $\lambda$)
 \begin{align}
   \label{eq:evi}
   \frac{1}{2}\frac{\dn^+}{\dn t}\bd^2(\flow^t(u),\tilde u)+\frac{\lambda}{2}\bd^2(\flow^t(u),\tilde u)+\mathcal{A}(\flow^t(u))
   &\le\mathcal{A}(\tilde u),
\end{align}
 holds for arbitrary $u,\tilde u$ in the domain of $\mathcal{A}$, and for all $t\ge 0$.
\end{definition}
By \cite{daneri2008}, $\lambda$-geodesic convexity is implied by \eqref{eq:evi}. Henceforth, $\flow$ is called \emph{gradient flow} of $\mathcal{A}$ with respect to the distance $\bd$.\\

A possible method to construct a gradient flow is by means of the so-called \emph{minimizing movement scheme}.

\begin{definition}[Minimizing movement]
Given a step size $\tau>0$ and an initial value $u_\tau^0$, determine $u_\tau^{k}$ inductively for $k\in\N$ as minimizers of
\begin{align}
 \label{eq:minmov_gen}
 \mathcal{A}_\tau(u\,|\,u_\tau^{k-1})&:=\frac{1}{2\tau}\bd^2(u,u_\tau^{k-1})+\mathcal{A}(u),
\end{align}
which exist under suitable conditions on the functional $\mathcal{A}$.
Then define a \emph{time-discrete solution} $u_\tau:[0,\infty)\to \X$ by piecewise constant interpolation:
\begin{align}
 \label{eq:disc_sol}
u_\tau(t)&:=u_\tau^k\text{ for }t\in ((k-1)\tau,k\tau].
\end{align}
\end{definition}

\begin{thm}[Flow interchange lemma {\cite[Thm. 3.2]{matthes2009}}]\label{thm:flowinterchange}
Let $\mathcal{B}$ be a proper, lower semicontinuous and $\lambda$-geodesically convex functional on $(\X,\bd)$. Let furthermore $\mathcal{A}$ be another proper, lower semicontinuous functional on $(\X,\bd)$ such that $\dom(\mathcal{A})\subset \dom(\mathcal{B})$. Assume that, for arbitrary $\tau>0$ and $\tilde u\in \X$, the functional $\mathcal{A}_\tau(\cdot\,|\,\tilde u)$ possesses a minimizer $u$. Then, the following holds:
\begin{align*}
\mathcal{B}(u)+\tau \mathrm{D}^\mathcal{B}\mathcal{A}(u)+\frac{\lambda}{2}\bd^2(u,\tilde u)&\le \mathcal{B}(\tilde u).
\end{align*}
There, $\mathrm{D}^\mathcal{B}\mathcal{A}(u)$ denotes the \textit{dissipation} of the functional $\mathcal{A}$ along the gradient flow $\flow^{(\cdot)}_\mathcal{B}$ of the functional $\mathcal{B}$, i.e.
\begin{align*}
\mathrm{D}^\mathcal{B}\mathcal{A}(u):=\limsup_{h\searrow 0}\frac{\mathcal{A}(u)-\mathcal{A}(\flow^h_\mathcal{B}(u)    )   }{h}.
\end{align*}
\end{thm}

In order to define a time-continuous flow, passage to the limit $\tau\searrow 0$ is necessary. The following theorem provides a useful tool in this context.

\begin{thm}[Extension of the Aubin-Lions lemma {\cite[Thm. 2]{rossi2003}}]\label{thm:ex_aub}
Let $\ban$ be a Banach space and $\mathcal{A}:\,\ban\to[0,\infty]$ be lower semicontinuous and have relatively compact sublevels in $\ban$. Let furthermore $\W:\,\ban\times \ban\to[0,\infty]$ be lower semicontinuous and such that $\W(u,\tilde u)=0$ for $u,\tilde u\in \dom(\mathcal{A})$ implies $u=\tilde u$.

Let $(U_k)_{k\in\N}$ be a sequence of measurable functions $U_k:\,(0,T)\to \ban$. If
\begin{align}
\sup_{k\in\N}\int_0^T\mathcal{A}(U_k(t))\dd t&<\infty,\label{eq:hypo1}\\
\lim_{h\searrow 0}\sup_{k\in\N}\W(U_k(t+h),U_k(t))\dd t&=0,\label{eq:hypo2}
\end{align}
then there exists a subsequence that converges in measure w.r.t. $t\in(0,T)$ to a limit $U:\,(0,T)\to \ban$.
\end{thm}


\section{Examples of mobility matrices}\label{sec:mobex}
This section is devoted to examples of mobility matrices $\M:\,S\to\Matn$ 
that satisfy conditions (C0)--(C3) stated in the introduction.
We will occasionally also consider the following stronger version of (C2):
\begin{enumerate}[(C1')]
  \setcounter{enumi}{1}
\item The matrix $\dff^2\M(z)[\zeta,\zeta]\in\Matn$ is negative definite for all $z\in \inn{S}$ and $\zeta\in\R^n\backslash\{0\}$.
\end{enumerate}
All of our examples are of the form \eqref{eq:Minduce}, where $\M$ is induced by a convex function $h$.

\subsection{Fully decoupled mobilities}\label{subsec:decoup}
Consider concave functions $\mob_1,\ldots,\mob_n$ with $\mob_j:\,[S^\ell_{j},S^r_j]\to\R$, $S^\ell_{j}<S^r_j$, such that $\mob_j(s)>0$ for $s\in (S^\ell_{j},S^r_j)$ and $\mob_j(S^\ell_j)=\mob_j(S^r_j)=0$, for each $j$.
Define a mobility matrix $\M:\,S\to\Matn$ on the $n$-cuboid $S:=[S^\ell,S^r]$ 
by 
\begin{align}
  \label{eq:decoupled}
  \M(z) =
  \begin{pmatrix}
    \mob_1(z_1) & & \\ & \ddots & \\ & & \mob_n(z_n)
  \end{pmatrix}.
\end{align}
Clearly $\M$ is of the form \eqref{eq:Minduce}, where
\begin{align*}
  h(z) = h_1(z_1) + \cdots + h_n(z_n),
\end{align*}
and each $h_j:(S^\ell_{j},S^r_j)\to\R$ is a second primitive of the respective $\frac1{\mob_j}$, i.e., $\mob_j(s)h_j''(s)=1$.
It is immediately verified that $\M$ satisfies (C0)--(C3).
Concerning property (C2), we remark that
\begin{align*}
  \beta^\tT\dff^2\M(z)[\zeta,\zeta]\beta = \sum_{j=1}^n \mob_j''(z)\,(\zeta_j\beta_j)^2,
\end{align*}
hence the sharper condition (C2') is \emph{not} satisfied, 
even if all $\mob_j$ are \emph{strictly} concave functions.
This is the reason why the concavity (C2) is lost under generic perturbations of $\M$.
In the next example below, we discuss 
a very special ``perturbation'' of a particular matrix of type \eqref{eq:decoupled},
for which (C2') is valid.

For obvious reasons, we call mobility matrices $\M$ of the form \eqref{eq:decoupled} \emph{fully decoupled}:
the different species do not influence each other's mobility.
It is clear that each fully decoupled matrix $\M$ induces a metric on $\measm$,
simply applying the theory from \cite{dns2009,lisini2010} to each component separately.

\subsection{Perturbations of a fully decoupled mobility}\label{subsec:pdec}
Let us now specialize the previous example by choosing $n=2$ components, $S=[0,1]^2$
and  $h_0:\,(0,1)^2\to\R$ with
\begin{align}
\label{eq:hnull}
  h_0(z) = z_1\log z_1 + (1-z_1)\log(1-z_1) + z_2\log z_2 + (1-z_2)\log(1-z_2).
\end{align}
From \eqref{eq:Minduce}, we obtain the fully decoupled mobility matrix
\begin{align*}
  \M_0(z) = \big(\hess h_0(z)\big)^{-1} = \begin{pmatrix} d_1 & 0 \\ 0 & d_2 \end{pmatrix},
  \quad \text{with} \quad
  d_j=\mob(z_j),
\end{align*}
where $\mob(s)=s(1-s)$.
By the discussion above, (C0)--(C3) are satisfied, but (C2') is not. It is easily seen that for a general (smooth, compactly supported) function $g:\,(0,1)^2\to\R$,
the matrix $\tilde\M_\eps=(\hess(h_0+\eps g))^{-1}$ does \emph{not} satisfy (C2) anymore, no matter how small $\eps>0$ is.

Let us introduce a very special perturbation $h_\eps$ of $h_0$:
\begin{align}
  \label{eq:heps}
  h_\eps(z) = h_0(z) + \eps z_1z_2(1-z_1)(1-z_2)=h_0(z) + \eps d_1d_2.
\end{align}
We are going to show that $\M_\eps(z)=(\bH_\eps(z))^{-1}$, 
with
\begin{align*}
  \bH_\eps(z):=\hess h_\eps(z)=
  \begin{pmatrix} \frac1{d_1} & 0 \\ 0 & \frac1{d_2} \end{pmatrix}
  + \eps \begin{pmatrix} -2d_2 & d_1'd_2' \\ d_1'd_2' & -2d_1 \end{pmatrix},
  \quad\text{with}\quad
  d_1'=1-2z_1,\ d_2'=1-2z_2,
\end{align*}
satisfies (C0)--(C3), and in addition also (C2'), for all sufficiently small $\eps>0$.
Thus, this special perturbation makes the mobility matrix robust with respect to further (smaller) generic perturbations.

First, note that $\M_\eps(z)$ is well-defined at $z\in(0,1)^2$ if 
\begin{align}
  \label{eq:Hdet}
  \det \bH_\eps (z) = \frac1{d_1d_2} - 4\eps + \eps^2\big(4d_1d_2-(d_1'd_2')^2\big)
\end{align}
is positive.
This is true simultaneously at all $z\in(0,1)^2$ if $\eps>0$ is sufficiently small.
It is further easily seen that $\M_\eps$ extends continuously to the boundary of $S$ by setting $\M_\eps(z)=\M_0(z)$ for $z\in\partial S$;
just observe that
\begin{align*}
  \M_\eps(z) = \frac1{1-\eps d_1d_2\big[4-\eps\big(4d_1d_2-(d_1'd_2')^2\big)\big]}
  \left[\M_0(z) - \eps d_1d_2 \begin{pmatrix} 2d_1 & d_1'd_2' \\ d_1'd_2' & 2d_2 \end{pmatrix}\right],
\end{align*}
and that $d_1d_2\searrow0$ as $z\to\partial S$.
This implies (C0) and (C3) for $\M_\eps$.

Next, since the entries of $\M_\eps$ vary continuously with $\eps$,
and since $\det\M_\eps(z)=(\det \bH_\eps(z))^{-1}$ never vanishes for any $z\in(0,1)^2$ and any sufficiently small $\eps>0$,
it follows that $\M_\eps$ inherits the positive definiteness of $\M_0$.
Thus, also (C1) is verified.

The proof of condition (C2') is more involved.
To begin with, observe that $\M_\eps(z) = (\bH_\eps(z))^{-1}$ implies
\begin{align*}
  \dff^2\M_\eps(z)[\zeta,\zeta] 
  = -\bH_\eps(z)^{-1}\dff^2\bH_\eps(z)[\zeta,\zeta]\bH_\eps(z)^{-1} + 2\bH_\eps(z)^{-1}\dff \bH_\eps(z)[\zeta]\bH_\eps(z)^{-1}\dff \bH_\eps(z)[\zeta]\bH_\eps(z)^{-1}.
\end{align*}
Thus, for proving (C2'), it suffices to show that
for all $z\in(0,1)^2$ and all $\beta,\zeta\in\R^n\setminus\{0\}$,
\begin{align*}
  P:= (d_1d_2)^3 \beta^\tT\big(\det \bH_\eps(z)\,\dff^2\bH_\eps(z)[\zeta,\zeta]-2\dff \bH_\eps(z)[\zeta]\bH_\eps(z)^{+}\dff \bH_\eps(z)[\zeta]\big)\beta > 0,
\end{align*}
where $\det \bH_\eps$ is given in \eqref{eq:Hdet}, and
\begin{align*}
  \dff \bH_\eps(z)[\zeta] &= \begin{pmatrix} -\frac{d_1'}{d_1^2}\zeta_1 & 0 \\ 0 & -\frac{d_2'}{d_2^2}\zeta_2 \end{pmatrix}
  -2 \eps \begin{pmatrix} d_2'\zeta_2 & d_1'\zeta_1+d_2'\zeta_1\\ d_1'\zeta_2+d_2'\zeta_1 & d_1'\zeta_1 \end{pmatrix}, \\
  \dff^2\bH_\eps(z)[\zeta,\zeta] &= 2\begin{pmatrix} \frac{1-3d_1}{d_1^3}\zeta_1^2 & 0 \\ 0 & \frac{1-3d_2}{d_2^3}\zeta_2^2 \end{pmatrix}
  + 4 \eps \begin{pmatrix} \zeta_2^2 & 2\zeta_1\zeta_2 \\ 2\zeta_1\zeta_2 & \zeta_1^2 \end{pmatrix}, \\
  \bH_\eps(z)^+ &= \det \bH_\eps(z)\,\bH_\eps(z)^{-1} 
  =  \begin{pmatrix} \frac1{d_2} & 0 \\ 0 & \frac1{d_1} \end{pmatrix}
  + \eps \begin{pmatrix} -2d_1 & - d_1'd_2' \\ - d_1'd_2' & -2d_2 \end{pmatrix}.
\end{align*}
A tedious but straightforward calculation leads to the following explicit representation of $P$, 
with the abbreviations $\tilde\zeta_1:=d_2\zeta_1$, $\tilde\zeta_2:=d_1\zeta_2$:
\begin{align*}
  P =& 2\big[\tilde\zeta_1^2+2\eps (d_2\tilde\zeta_2)^2\big]\beta_1^2 + 2\big[\tilde\zeta_2^2+2\eps (d_1\tilde\zeta_1)^2\big]\beta_2^2+ \eps\big[\hat f_1\tilde\zeta_1^2+\eps f_2(d_2\tilde\zeta_2)^2+\hat f_3\xi^1(d_2\tilde\zeta_2)\big]\beta_1^2 \\
  &+ \eps\big[\check f_1\tilde\zeta_2^2+\eps f_2(d_1\tilde\zeta_1)^2+\check f_3\tilde\zeta_2(d_1\tilde\zeta_1)\big]\beta_2^2 + \eps\big[\hat f_4 \tilde\zeta_1(d_1\tilde\zeta_1) + \check f_4 \tilde\zeta_2(d_2\tilde\zeta_2) + 2f_5\tilde\zeta_1\tilde\zeta_2\big]\beta_1\beta_2,
\end{align*}
where the functions $f_i$, $\hat f_i$ and $\check f_i$ are bounded, uniformly with respect to $z\in(0,1)^2$ and (small) $\eps>0$:
\begin{align*}
  \hat f_1&:=2d_2(8d_1-3)+\eps(d_2(-64d_1+132d_1^2+8)-39d_1^2-2+18d_1)-8\eps^2d_1^3d_2^2(1-4d_2),\\
  \check f_1&:=2d_1(8d_2-3)+\eps(d_1(-64d_2+132d_2^2+8)-39d_2^2-2+18d_2)-8\eps^2d_2^3d_1^2(1-4d_1),\\
  f_2&:=4d_1d_2-d_1-d_2-\eps d_1d_2(28d_1d_2-10d_1-10d_2+3),\\
  \hat f_3&:=4d_1'd_2'\big(1+\eps d_2(1-2\eps(1-2d_2)d_1^2)\big),\\
  \check f_3&:=4d_1'd_2'\Big(1+\eps d_1(1-2\eps(1-2d_1)d_2^2)\big),\\
  \hat f_4&:=-2d_1'd_2'\big(1+4\eps(d_1-1+2\eps d_2d_1^2(1-2d_2))\big),\\
  \check f_4&:=-2d_1'd_2'\big(1+4\eps(d_2-1+2\eps d_1d_2^2(1-2d_1))\big),\\
  f_5&:=1+40d_1d_2-6d_1-6d_2-32\eps d_1^2d_2^2-4\eps^2d_1^2d_2^2(5+56d_1d_2-18d_1-18d_2).
\end{align*}
From elementary calculations -- applying the Cauchy-Schwarz inequality and collecting terms --
we conclude that 
\begin{align*}
  P \ge \big[\tilde\zeta_1^2+2\eps (d_2\tilde\zeta_2)^2\big]\beta_1^2 + \big[\tilde\zeta_2^2+2\eps (d_1\tilde\zeta_1)^2\big]\beta_2^2
\end{align*}
for arbitrary $z\in(0,1)^2$, $\beta,\zeta\in\R^n$, and all sufficiently small $\eps>0$.
This implies positivity of $P$ for $\beta\neq0$ and $\zeta\neq0$, and therefore proves (C2').

\subsection{Volume filling mobility}\label{subsec:volume}
The following example describes the interaction of species 
that influence each other's mobilities by competing for limited volume.
This example is related but not identical to the one considered in \cite{liero2012},
where in addition a microlocal conservation of mass was assumed.

Define the \emph{$n^{th}$ standard simplex}
\begin{align*}
  S:=\left\{z\in [0,1]^n:\,1-\sum_{j=1}^n z_j\ge 0\right\}
\end{align*}
as state space and the map $h:\,\inn{S}\to\R$ by
\begin{align*}
  h(z):=\sum_{j=1}^n z_j\log z_j+\left(1-\sum_{j=1}^n z_j\right)\log\left(1-\sum_{j=1}^n z_j\right).
\end{align*}

The second order derivatives of $h$ amount to
\begin{align*}
  \frac{\partial^2h}{\partial z_i\,\partial z_j}(z)=\frac1{z_i}\delta_{ij}+\frac1{1-\sum_{\ell=1}^n z_\ell},
\end{align*}
where $\delta_{ij}$ denotes Kronecker's delta. 
By elementary calculations, we obtain the explicit form of the inverse matrix,
\begin{align*}
  \M(z) = (\hess h(z))^{-1} 
  = \begin{pmatrix} z_1 & & \\ & \ddots & \\ & & z_n \end{pmatrix} - zz^\tT.
\end{align*}
Property (C0) obviously holds.
To verify (C1), let $\gamma\in\R^n$ be given 
and observe that
\begin{align*}
  \gamma^\tT zz^\tT \gamma = \sum_{i,j=1}^n z_iz_j\gamma_i\gamma_j \le \frac12\sum_{i,j=1}^n z_iz_j(\gamma_i^2+\gamma_j^2) 
  = \left(\sum_{j=1}^n \gamma_j^2z_j\right)\left(\sum_{\ell=1}^n z_\ell\right).
\end{align*}
Therefore,
\begin{align*}
  \gamma^\tT\M(z)\gamma = \sum_{j=1}^n \gamma_j^2z_j - \gamma^\tT zz^\tT \gamma\ge \sum_{j=1}^k \gamma_j^2z_j\left(1-\sum_{\ell=1}^nz_\ell\right),
\end{align*}
which is positive for all $z\in\inn{S}$ and $\gamma\neq0$.
Condition (C2) is immediately obtained from
\begin{align*}
  \dff^2\M(z)[\zeta,\zeta] = -2\zeta\zeta^\tT,
\end{align*}
which is negative semidefinite, for arbitrary $z\in\inn{S}$ and $\zeta\in\R^n$.
Note that $\dff^2\M(z)[\zeta,\zeta]$ has rank one, hence the stronger condition (C2') is not satisfied.
Finally, let $z\in\partial S$, and let $\nu$ be a normal vector to $\partial S$ at $z$.
We distinguish two cases. 
In the first, $z$ lies on one of the coordinate hyperplanes. 
Then $\nu_j\neq0$ only if $z_j=0$, for $j=1,\ldots,n$, and so clearly $\M(z)\nu=0$.
In the second case, we have $z_1+\cdots+z_n=1$. 
Hence the normal vector is (a multiple of) $\einsvec=(1,\ldots,1)^\tT$,
and therefore
\begin{align*}
  \M(z)\einsvec = z - z(z^\tT\einsvec) = \left(1-\sum_{\ell=1}^nz_\ell\right)\,z = 0.
\end{align*}
This proves (C3).

\subsection{Radially symmetric mobility}\label{subsec:radialmob}
On the $n$-dimensional closed unit ball $S:=\overline{\ball_1(0)}$, 
define $h:\,S\to\R$ by
\begin{align*}
  h(z) = \log\big(1+\sqrt{1-|z|^2}\big) - \sqrt{1-|z|^2}.
\end{align*}
One easily verifies that
\begin{align*}
  \hess h(z) = \frac1{1+\sqrt{1-|z|^2}}\Id + \frac1{\big(1+\sqrt{1-|z|^2}\big)\sqrt{1-|z|^2}}\frac{zz^\tT}{|z|^2},
\end{align*}
which obviously is positive definite for $z\in\ball_1(0)$.
Now define $\M$ by \eqref{eq:Minduce}, i.e.,
\begin{align*}
  \M(z) &= \big(\hess h(z)\big)^{-1}\\
  &= \big(1+\sqrt{1-|z|^2}\big)\Id 
  + \big[\big(1+\sqrt{1-|z|^2}\big)\sqrt{1-|z|^2}-\big(1+\sqrt{1-|z|^2}\big)\big]\frac{zz^\tT}{|z|^2} \\
  &= \big(1+\sqrt{1-|z|^2}\big)\Id - zz^\tT.
\end{align*}
Conditions (C0) and (C1) obviously hold.
Next, for arbitrary $\zeta\in\R^n$ and $z\in\inn{S}$, 
we have that
\begin{align*}
  \dff\M(z)[\zeta] &= - (1-|z|^2)^{-1/2}(z^\tT \zeta)\Id - z\zeta^\tT - \zeta z^\tT \\ 
  \dff^2\M(z)[\zeta,\zeta] &= -(1-|z|^2)^{-3/2}(z^\tT \zeta)^2\Id  - 2\zeta\zeta^\tT- (1-|z|^2)^{-1/2}|\zeta|^2\Id.
\end{align*}
The last matrix is obviously negative definite for each $z\in\ball_1(0)$ and $\zeta\neq 0$, which shows (C2').
Finally, to verify (C3), 
let $z\in S$ with $|z|=1$ be given, and observe that $z$ itself is a normal vector to $\partial S$ at $z$.
One has
\begin{align*}
  \M(z)z = (1+\sqrt{1-|z|^2})z-|z|^2z = z-z=0.
\end{align*}


\section{Distances generated by a mobility matrix}\label{sec:metric}
This section is devoted to the study of transport distances between vector-valued densities on $\R$.
Throughout this section, let some convex and compact set $S\subset\R^n$ with nonempty interior be fixed,
and recall that $\measm$ is the space of measurable functions on $\R$ with values in $S$.
Throughout this section, we assume that $\M:\,S\to\Matn$ is a mobility matrix that satisfies (C0)--(C3).

The object of central interest in this section is the function $\W_\M:\,\measm\times\measm\to[0,\infty]$,
defined by
\begin{align}
  \label{eq:def_WM}
  \W_\M(\mu_0,\mu_1)&:=\left[\inf\left\{\int_0^1\int_\R w^\tT  (\M(\mu))^{-1} w\dd x\dd t:\,
      (\mu,w)\in\scrC_1(\mu_0\to\mu_1)\right\}\right]^{1/2},
\end{align}
where $\scrC_1(\mu_0\to\mu_1)$ denotes the set of all curves $(\mu,w)=(\mu_t,w_t)_{t\in[0,1]}$
satisfying the \emph{continuity equation} 
\begin{align}
  \label{eq:conti}
  \partial_t \mu+\partial_x w=0,  
\end{align}
having $\mu_0$ and $\mu_1$ as starting and terminal values, respectively. We begin by giving a rigorous definition of the objects occurring above.

\subsection{The action density function}\label{subsec:def_metr}
\begin{prop}[Properties of the density function]
  \label{prop:prop_dens}
  The \emph{action density function} $\tilde\phi:\,\inn{S}\times\R^{n}\to[0,\infty)$, defined by
  \begin{align}
    \label{eq:density}
    \tilde\phi(z,p)&:=p^\tT (\M(z))^{-1}p
  \end{align} 
  has the following properties:
  \begin{enumerate}[(a)]
  \item $\tilde\phi$ is continuous and (jointly) convex.
  \item $\tilde\phi$ is nondegenerate: $\tilde\phi(z,p)>0$ for all $z\in\inn{S}$, $p\neq 0$.
  \item $\tilde\phi$ is 2-homogeneous in its second component.
  \end{enumerate}
\end{prop}

\begin{proof}
  Since $\M$ is subject to (C0)--(C2), only convexity is not obvious. 
  For the second directional derivative of $\phi$ at $(z,p)$ in directions $(\zeta,\pi)$ for $\zeta\in\R^n$, $\pi\in\R^n$, 
  we obtain, 
  \begin{align}
    \label{eq:2ndderiv}
    \dff^2_{(z,p)}\tilde\phi(z,p)[(\zeta,\pi),(\zeta,\pi)] &= \pi^\tT  A\pi+p^\tT  B\pi+p^\tT  Cp,
  \end{align}
  with
\begin{align*}
A&:=2\M(z)^{-1},\\
B&:=-4\M(z)^{-1}\dff\M(z)[\zeta]\M(z)^{-1},\\
C&:=2\M(z)^{-1}\dff\M(z)[\zeta]\M(z)^{-1}\dff\M(z)[\zeta]\M(z)^{-1}-\M(z)^{-1}\dff^2\M(z)[\zeta,\zeta]\M(z)^{-1}.
\end{align*}

We prove that the expression in \eqref{eq:2ndderiv} is nonnegative, for all admissible choices of $(z,p)$ and $(\zeta,\pi)$.
  Since, by condition (C1), $A$ is symmetric positive definite, there exists a symmetric positive definite square root $A^{1/2}\in\Matn$ 
  such that $A^{1/2}A^{1/2}=A$.
  Further, $B$ is symmetric. By elementary calculation, we obtain

\begin{align*}
 \dff^2_{(z,p)}\tilde\phi(z,p)[(\zeta,\pi),(\zeta,\pi)] &=\left|A^{1/2}\pi+\frac12 A^{-1/2}B  p\right|^2+\frac14 p^\tT (4C-BA^{-1}B )p\\
&=\left|A^{1/2}\pi+\frac12 A^{-1/2}B  p\right|^2-p^\tT\M(z)^{-1}\dff^2\M(z)[\zeta,\zeta]\M(z)^{-1}p,
\end{align*}
which is nonnegative due to condition (C2).
\end{proof}

Below, we need the action density to be defined up to the boundary.
To this end, we replace $\tilde\phi$ by its \emph{lower semicontinuous envelope} $\phi:\,S\times\R^{n}\to[0,\infty]$,
defined by
\begin{align}
  \label{eq:lscphi}
  \phi(\tilde z,\tilde p):=\liminf_{(z,p)\to(\tilde z,\tilde p)}\tilde\phi(z,p).
\end{align}
Thanks to continuity of $\tilde\phi$, we have $\phi\equiv\tilde\phi$ on $\inn{S}\times\R^{n}$.
\begin{xmp}
  Let $\M(z)=(\hess h_\eps(z))^{-1}$ with $h$ given in \eqref{eq:heps}.
  Then, for $z=(z_1,z_2)\in[0,1]^2$ and every $p=(p_1,p_2)\in\R^2$, 
  we have that
  \begin{align*}
    \phi(z,p)&=\begin{cases}
      p^\tT\hess h_\eps(z)p,&\text{if }z\in(0,1)^2,\\
      \frac{p_2^2}{z_2(1-z_2)},&\text{if } z_1\in\{0,1\},\,z_2\in(0,1),\,p=(0,p_2),\\
      \frac{p_1^2}{z_1(1-z_1)},&\text{if } z_2\in\{0,1\},\,z_1\in(0,1),\,p=(p_1,0),\\
      0,&\text{if }z\in\{(0,0),(1,0),(1,1),(0,1)\}\text{ and }p=0,\\
      +\infty&\text{otherwise}.
    \end{cases}
  \end{align*}
  The key step in the derivation is to observe that 
  if $z$ tends to a boundary point $\tilde z\in\partial S$ that is not a corner,
  then precisely one of the two eigenvalues of $\hess h(z)$ converges to zero, 
  and the eigenvector for the non-vanishing eigenvalue is asymptotically parallel to $\partial S$ at $\tilde z$.
\end{xmp}
For $\mu\in\measm$ and $w\in\meas$, we define the \emph{action functional}
\begin{align}
  \label{eq:action}
  \act(\mu,w)&:=\int_{\R}\phi\big(\mu(x),w(x)\big)\dd x.
\end{align}
Proposition \ref{prop:prop_dens} allows to apply Theorem 2.1 in \cite{dns2009} to obtain:
\begin{prop}[Lower semicontinuity of the action functional]
  If $(\mu_k)_{k\in\N}$ and $(w_k)_{k\in\N}$ are weakly$\ast$-convergent sequences to $\mu\in\measm$ and $w\in\meas$, respectively, 
  then 
  \begin{align*}
    \liminf_{k\to\infty}\act(\mu_k,w_k)&\ge \act(\mu,w).
  \end{align*}
\end{prop}

\subsection{Solutions to the continuity equation}\label{subsec:sol_cont}
Next, we investigate the structure of solutions to the (multi-component) \emph{continuity equation} \eqref{eq:conti}.

Since the components of $\mu$ and $w$ are decoupled in \eqref{eq:conti}, 
most of the results below follow from a ``component-wise application'' of the corresponding results in \cite{dns2009,lisini2010}.

\begin{definition}[The class $\scrC_T$]\label{def:CT}
  Given $T>0$, define $\scrC_T$ as the set of all curves $(\mu,w)=(\mu_t,w_t)_{t\in[0,T]}$ with the following properties:
  \begin{enumerate}[(a)]
  \item $(\mu_t)_{t\in[0,T]}$ is a weakly$\ast$-continuous curve in $\measm$,
  \item $(w_t)_{t\in[0,T]}$ is a Borel-measurable family in $\meas$.
  \item For each $R>0$, 
  \begin{align*}
    \int_0^T\int_{-R}^R |w_t|\dd x\dd t<\infty,
  \end{align*}
  \item $(\mu,w)$ is a distributional solution to \eqref{eq:conti} on $[0,T]\times\R$.
  \end{enumerate}
  Furthermore, we denote by $\scrC_T(\hat\mu\to \check\mu)$ the subset of those $(\mu,w)\in\scrC_T$ 
  with $\mu|_{t=0}=\hat\mu$ and $\mu|_{t=T}=\check\mu$.
\end{definition}
The continuity property (a) above imposes no restriction on the curve $(\mu_t,w_t)_{t\in[0,T]}$.
Indeed, by componentwise application of Lemma 4.1 from \cite{dns2009},
one deduces that every $(\mu_t,w_t)_{t\in[0,T]}$ satisfying (b)--(d) possesses a uniquely determined weak$\ast$-continuous representative.
\begin{lemma}[Time rescaling]
  \label{lemma:scale}
  Let $\sigma:\,[0,T']\to[0,T]$ be almost everywhere equal to a diffeomorphism. 
  Then $(\mu,w)$ is a distributional solution of \eqref{eq:conti} on $[0,T]\times\R$ if and only if 
  $(\hat\mu,\hat w):=(\mu\circ\sigma,\sigma'\cdot w\circ\sigma)$ is a distributional solution of \eqref{eq:conti} on $[0,T']\times\R$.
\end{lemma}
\begin{proof}
  \cite[Lemma 8.1.3]{savare2008}.
\end{proof}
\begin{lemma}[Glueing lemma]
  \label{lemma:glue}
  Let $(\hat\mu,\hat w)\in\scrC_{T_1}(\mu_0\to\mu_1)$, $(\hat{\hat\mu},\hat{\hat w})\in\scrC_{T_2}(\mu_1\to\mu_2)$. 
  Then the concatenation $(\mu,w)=(\mu_t,w_t)_{t\in[0,T]}$ with $T=T_1+T_2$, 
  defined by
  \begin{align*}
    (\mu_t,w_t):=\begin{cases}
      (\hat\mu_t,\hat w_t)&\text{for }t\in[0,T_1],\\
      (\hat{\hat\mu}_{t-T_1},\hat{\hat w}_{t-T_1})&\text{for }t\in(T_1,T_1+T_2],
    \end{cases}
  \end{align*}
  is an element of $\scrC_{T}(\mu_0\to\mu_2)$.
\end{lemma}
\begin{proof}
  This is a direct consequence of Lemma \ref{lemma:scale}, see for instance \cite{dns2009}.
\end{proof}
\begin{definition}
  The \emph{energy} $\E_T$ of a curve $(\mu,w)=(\mu_t,w_t)_{t\in[0,T]}\in\scrC_T$ is defined by
  \begin{align*}
    \E_T(\mu,w):=\int_0^T\act(\mu_t,w_t)\dd t.
  \end{align*}  
\end{definition}

\begin{prop}[Compactness in $\scrC_T$, part I]
  \label{prop:compCT}
  Let $(\mu_k,w_k)_{k\in\N}$ be a sequence in $\scrC_T$ such that for each fixed $R>0$,
  the family 
\begin{align*}
\left\{t\mapsto\int_{-R}^R|(w_k)_t|\dd x:\,k\in\N\right\}
\end{align*}
of maps from $(0,T)$ into $\R^n$ is $k$-uniformly integrable.
  Then, there exists a subsequence (non-relabelled) and $(\mu,w)\in\scrC_T$ 
  such that for $k\to\infty$:
  \begin{align*}
    &(\mu_k)_t\stackrel{\ast}{\rightharpoonup}\mu_t \text{ weakly$\ast$ in }\measm\text{ for every }t\in[0,T],\\
    &w_k\stackrel{\ast}{\rightharpoonup}w\text{ weakly$\ast$ in }\meas,\\ 
    &\E_T(\mu,w)\le\liminf_{k\to\infty}\E_T(\mu_k,w_k).
  \end{align*}
\end{prop}
\begin{proof}
  Apply Lemma 4.5 in \cite{dns2009} componentwise.
\end{proof}

\begin{prop}[Compactness in $\scrC_T$, part II]
  \label{prop:compCTE}
  Let $(\mu_k,w_k)_{k\in\N}$ be a sequence in $\scrC_T$ of uniformly bounded energy,
  \begin{align*}
    \sup_{k\in\N}\E_T(\mu_k,w_k)&<\infty.
  \end{align*}
  Then the hypotheses of Proposition \ref{prop:compCT} are fulfilled.
\end{prop}
\begin{proof}
  To begin with, observe that thanks to continuity of $\M$ by (C0),
  there exists a constant $C_\M>0$ such that $\|\M(z)\|\le C_\M$ for all $z\in S$.
  Hence $\phi(z,p)\ge C_\M^{-1}|p|^2$, for all $(z,p)\in S\times\R^{ n}$.
  For given $R>0$, we have that:
  \begin{align*}
    &\int_0^T\bigg|\int_{-R}^R|(w_k)_t|\dd x\bigg|^2\dd t=\sum_{j=1}^n\int_0^T\left[\int_{-R}^R|w_{k,j}|_t\dd x\right]^2\dd t\\
&\le \sum_{j=1}^n\int_0^T 4R^2\int_\R |w_{k,j}|_t^2\dd x\dd t\le 4 R^2 C_\M\int_0^T\int_\R\phi(\mu_k,w_k)\dd x\dd t\le 4R^2C_\M\sup_{k\in\N}\E_T(\mu_k,w_k)\le C<\infty.
  \end{align*}
  This proves that the family $\int_{-R}^R|(w_k)_t|\dd x$ is $k$-uniformly bounded in $L^2((0,T);\R^n)$.
\end{proof}

\subsection{Distance functional and topological properties}\label{subsec:topo}
In this section, we prove that the distance $\W_\M$ with
\begin{align}
  \label{eq:def_WMabb}
  \W_\M(\mu_0,\mu_1)&:=\left[\inf\left\{\E_1(\mu,w):\,(\mu,w)\in\scrC_1(\mu_0\to\mu_1)\right\}\right]^{1/2}
\end{align}
is a (pseudo-) metric on $\measm$ and investigate topological properties of $\W_\M$.
\begin{prop}[Minimizers and equivalent characterization]
  \label{prop:minWM}
  The following statements hold:
  \begin{enumerate}[(1)]
  \item If the infimum $W$ occurring in $\W_\M$ is finite, then it is attained by a curve $(\mu,w)\in\scrC_1(\mu_0\to\mu_1)$, for which one has
    \begin{align*}
      \act(\mu_t,w_t)&=W\quad\text{for a.e. }t\in(0,1).
    \end{align*}
    Consequently,
    \begin{align*}
      \W_\M(\mu_s,\mu_t)&=|t-s|\W_\M(\mu_0,\mu_1)\quad\forall\,s,t\in[0,1].
    \end{align*}
  \item There are two equivalent characterizations of $\W_\M$: For all $T>0$, 
    \begin{align}
      \W_\M(\mu_0,\mu_1)&=\left[\inf\left\{T\E_T(\mu,w):\,(\mu,w)\in\scrC_T(\mu_0\to\mu_1)\right\}\right]^{1/2}\label{eq:equiv_1}\\
      &=\inf\left\{\int_0^T[\act(\mu_t,w_t)]^{1/2}\dd t:\,(\mu,w)\in\scrC_T(\mu_0\to\mu_1)\right\}.\label{eq:equiv_2}
    \end{align}
  \end{enumerate}
\end{prop}

\begin{proof}
  The proof of \eqref{eq:equiv_2} is essentially the same as in \cite{dns2009,savare2008}, using the Rescaling Lemma \ref{lemma:scale}. 
  The other characterization \eqref{eq:equiv_1} can also be obtained by this lemma using a linear rescaling of time.

  For the proof of statement (1), assume that $\W_\M(\mu_0,\mu_1)=W^{1/2}<\infty$ for $W\ge 0$. 
  Then, there exists a sequence $(\mu_k,w_k)_{k\in\N}$ in $\scrC_1(\mu_0\to\mu_1)$ with $\sup\limits_{k\in\N}\E_1(\mu_k,w_k)<\infty$. 
  The application of the Propositions \ref{prop:compCTE} and \ref{prop:compCT} yields a limit curve $(\mu,w)\in \scrC_1(\mu_0\to\mu_1)$ 
  that is a minimizer of $\E_1$ on $\scrC_1(\mu_0\to\mu_1)$ due to weak$\ast$-lower semicontinuity. 
  With \eqref{eq:equiv_2}, one deduces
  \begin{align*}
    W^{1/2}&=\int_0^1\act(\mu_t,w_t)^{1/2}\dd t,
  \end{align*}
  and consequently, since $(0,1)\ni t\mapsto \act^{1/2}(\mu_t,w_t)$ and $(0,1)\ni t\mapsto 1$ yield equality in Hölder's inequality, 
  $\act(\mu_t,w_t)=W$ for almost every $t\in(0,1)$.
\end{proof}
We are now in position to prove that $\W_\M$ is a distance.
\begin{prop}[$\W_\M$ is a pseudometric]
  $\W_\M$ is a (possibly $\infty$-valued) metric on the space $\measm$.
\end{prop}
\begin{proof}
  \begin{itemize}
  \item Symmetry: This is immediate from the 2-homogeneity of $\phi$ and the Rescaling Lemma \ref{lemma:scale}. 
  \item Definiteness: $\W_\M(\mu_0,\mu_1)=0$ if and only if $\E_1(\mu,w)=0$ for some $(\mu,w)\in\scrC_1(\mu_0\to\mu_1)$. 
    From positive definiteness of $\M$, this is the case if and only if $w\equiv 0$ for some $(\mu,w)\in\scrC_1(\mu_0\to\mu_1)$, hence iff $\mu_0=\mu_1$.
  \item Triangle inequality: Let $\mu_0,\mu_1,\mu_2\in\measm$. If $\W_\M(\mu_0,\mu_1)$ or $\W_\M(\mu_1,\mu_2)$ is equal to $+\infty$, there is nothing to prove. 
    If both are finite, we can use the second equivalent characterization of $\W_\M$ \eqref{eq:equiv_2} and the Glueing Lemma \ref{lemma:glue} 
    to obtain $(\mu,w)\in\scrC_1(\mu_0\to\mu_1)$ such that $\W_\M(\mu_0,\mu_1)+\W_\M(\mu_1,\mu_2)=\int_0^1\act(\mu_t,w_t)^{1/2}\dd t$. 
    Again, invoking \eqref{eq:equiv_2}, we obtain the triangle inequality.
  \end{itemize}
\end{proof}
The following topological results are a consequence of the compactness results of Section \ref{subsec:sol_cont}, in particular of Proposition \ref{prop:compCTE}.
\begin{prop}[Topological properties]
  \label{prop:topo}
  The following statements hold:
  \begin{enumerate}[(a)]
  \item $\W_\M$ is lower semicontinuous in both components with respect to weak$\ast$-convergence.
  \item Let $\mu_0\in\measm$ fixed, but arbitrary and let $K\subset\measm$. If there exists $C\in\R$ such that $\W_\M(\mu_0,\mu)\le C$ for all $\mu\in K$, then $K$ is relatively compact in the weak$\ast$ topology.
  \item Let $\mu_0\in\measm$ fixed, but arbitrary and define $\X[\mu_0]:=\{\mu\in \measm:\,\W_\M(\mu_0,\mu)<\infty\}$. Then, the metric space $(\X[\mu_0],\W_\M)$ is complete.
\item $\W_\M^2$ is convex with respect to the linear structure of $\measm$: If $\mu_0$, $\mu_1$, $\tilde\mu_0$, $\tilde\mu_1\in\measm$ and $\tau\in [0,1]$, then
\begin{align*}
\W_\M^2((1-\tau)\mu_0+\tau\tilde\mu_0,(1-\tau)\mu_1+\tau\tilde\mu_1)&\le(1-\tau)\W_\M^2(\mu_0,\mu_1)+\tau \W_\M^2(\tilde\mu_0,\tilde\mu_1).
\end{align*}
\item Let $\Gamma\in C^\infty(\R)$ be nonnegative, with support in $[-1,1]$ and $\|\Gamma\|_{L^1}=1$, and let $\Gamma_\eps(x):=\frac1{\eps}\Gamma\left(\frac{x}{\eps}\right)$ for $\eps>0$. For all $\mu_0,\mu_1\in\measm$, the following holds:
\begin{align*}
\W_\M(\mu_0\ast \Gamma_\eps,\mu_1\ast\Gamma_\eps)&\le \W_\M(\mu_0,\mu_1),\\
\lim_{\eps\to 0}\W_\M(\mu_0\ast \Gamma_\eps,\mu_1\ast\Gamma_\eps)&=\W_\M(\mu_0,\mu_1).
\end{align*} 
  \end{enumerate}
\end{prop}

\begin{proof}
\begin{enumerate}[(a)]
\item Let $(\mu_{0,k},\mu_{1,k})_{k\in\N}$ be weakly$\ast$ convergent to $(\mu_0,\mu_1)$ as $k\to\infty$. Without loss of generality, there exists $Z\ge 0$ such that $\sup\limits_{k\in\N}\W_\M(\mu_{0,k},\mu_{1,k})\le Z$. From Proposition \ref{prop:minWM}(1), we obtain a sequence $(\mu_k,w_k)_{k\in\N}$ with $(\mu_k,w_k)\in\scrC_1(\mu_{0,k}\to\mu_{1,k})$ such that $\W_\M^2(\mu_{0,k},\mu_{1,k})=\act((\mu_k)_t,(w_k)_t)\le Z^2$ for almost every $t\in [0,1]$ and all $k\in\N$. Hence, the requirement of Proposition \ref{prop:compCTE} is fulfilled. The application of this proposition together with Proposition \ref{prop:compCT} now yields a limit curve $(\mu,w)\in\scrC_1(\mu_0\to\mu_1)$ and
\begin{align*}
\liminf_{k\to\infty}\W_\M^2(\mu_{0,k},\mu_{1,k})&=\liminf_{k\to\infty} \E_1(\mu_k,w_k)\ge \E_1(\mu,w)\ge \W_\M^2(\mu_0,\mu_1).
\end{align*}
\item If there exists $C\in\R$ such that $\W_\M(\mu_0,\mu)\le C$ for all $\mu\in K$, we can find by Proposition \ref{prop:minWM}(1) for each $k\in\N$ a curve $((\mu_k)_t,(w_k)_t)_{t\in[0,1]}$ in $\scrC_1(\mu_0\to \mu_k)$ such that $\act((\mu_k)_t,(w_k)_t)\le C^2$ for a.e. $t\in[0,1]$ and all $k\in\N$. The requirement of Proposition \ref{prop:compCTE} is again fulfilled. Its application yields in particular that $(\mu_k)_t\stackrel{\ast}{\rightharpoonup}\mu_t$ (on a subsequence) for all $t\in[0,1]$ and some $(\mu_t)_{t\in[0,1]}$. 
\item This proof is analogous to the proof of \cite[Thm. 5.7]{dns2009} using (a) and (b) of this proposition.
\item This is a consequence of convexity of the action density $\phi$.
\item This statement can be obtained as in \cite[Thm. 5.15]{dns2009}.
\end{enumerate}
\end{proof}

\begin{prop}[Smooth approximation of geodesics]\label{prop:smooth}
Assume that the mobility $\M$ is induced by $h:\,S\to\R$. Let $\mu_0,\mu_1\in\measm$ at finite distance $\W_\M(\mu_0,\mu_1)<\infty$ and such that for $i\in\{0,1\}$: 
\begin{align}
\label{eq:addsmooth}
\lim_{\delta\searrow 0}\delta\int_\R \left[h(\mu_i)- h(\krnl_\delta\ast \mu_i)\right]\dd x&=0,
\end{align}
where $\krnl_\delta$ denotes the \emph{heat kernel}
\begin{align}
\label{eq:heatkern}
\krnl_s(y):=\frac1{\sqrt{4\pi s}}\exp\left(-\frac{y^2}{4s}\right),\quad y\in\R,\,s>0.
\end{align}
Then, a geodesic curve $(\mu_s)$ connecting $\mu_0$ and $\mu_1$ given by Proposition \ref{prop:minWM} can be approximated with respect to the distance $\W_\M$ by the smooth curve $(\krnl_\delta\ast\mu_s)$, as $\delta\searrow 0$.
\end{prop}

\begin{proof}
For each $\delta>0$, we use the triangle inequality
\begin{align*}
\W_\M(\mu_0,\mu_1)\le \W_\M(\mu_0,\krnl_\delta\ast \mu_0)+\W_\M(\krnl_\delta\ast\mu_0,\krnl_\delta\ast\mu_1)+\W_\M(\mu_1,\krnl_\delta\ast\mu_1),
\end{align*}
and prove that the right-hand side converges to $\W_\M(\mu_0,\mu_1)$ as $\delta\searrow 0$. Adaptation of the proof of Lemma 8.1.9 in \cite{savare2008} to our setting immediately yields that 
\begin{align*}
\W_\M(\krnl_\delta\ast\mu_0,\krnl_\delta\ast\mu_1)\to \W_\M(\mu_0,\mu_1).
\end{align*}
Consider now the first term above. Define for $t\in[0,1]$
\begin{align*}
\tilde\mu_t&:=\krnl_{\delta t}\ast \mu_0,\\
\tilde w_t&:=-\delta\partial_x(\krnl_{\delta t}\ast\mu_0).
\end{align*}
Due to the smoothing property of the heat kernel, it is obvious that $(\tilde\mu,\tilde w)\in\scrC_1(\mu_0\to\krnl_{\delta t}\ast \mu_0)$. For the energy of this particular curve, we obtain thanks to $\partial_x(\grd h(\krnl_{\delta t}\ast\mu_0))=\hess h(\krnl_{\delta t}\ast\mu_0)\partial_x(\krnl_{\delta t}\ast\mu_0)$ and $(\M(z))^{-1}=\hess h(z)$:
\begin{align*}
\E_1(\tilde\mu,\tilde w)&=\delta^2\int_{[0,1]\times\R} \partial_x(\krnl_{\delta t}\ast\mu_0)^\tT \M^{-1}(\krnl_{\delta t}\ast\mu_0)\partial_x(\krnl_{\delta t}\ast\mu_0)\dd (t,x)\\
&=-\delta^2\int_{[0,1]\times\R}  \grd h(\krnl_{\delta t}\ast\mu_0)^\tT\partial_{xx}(\krnl_{\delta t}\ast\mu_0) \dd (t,x)\\
&=-\delta\int_{[0,1]\times\R}  \partial_t h(\krnl_{\delta t}\ast\mu_0)\dd (t,x)
=\delta\int_\R \left[h(\mu_0)- h(\krnl_\delta\ast \mu_0)\right]\dd x\rightarrow 0,
\end{align*}
proving the claim.
\end{proof}

\begin{remark}[Compactly supported velocity]
With additional technical effort, one can also prove that a solution curve $(\mu,\xi)=(\mu_t,\xi_t)_{t\in[0,1]}$ to the problem
\begin{align*}
\partial_t \mu+\partial_x (\M(\mu)\partial_x \xi)&=0,\quad\mu|_{t=0}=\mu_0,\quad \mu|_{t=1}=\mu_1,
\end{align*}
can be approximated by a smooth curve $(\tilde\mu,\tilde\xi)$, where $\partial_x\tilde\xi$ has compact support.
\end{remark}

Under specialized conditions, an estimate of $\W_\M$ in terms of the second moment $\boldsymbol{\ell}_2$ is possible:
\begin{prop}[Distance and second moment]\label{prop:moment}
Consider a state space $S\subset\R^n$ of the following form: There exists $S^\ell \in \partial S$ such that $z-S^\ell\ge 0$ (component-wise) for all $z\in S$. Assume that the mobility $\M$ satisfies, in addition to (C0)--(C3), the following Lipschitz-type condition w.r.t. $z$:
\begin{align}
\label{eq:lipmob}
\einsvec^\tT \M(z)\einsvec\le L \einsvec^\tT (z-S^\ell),\qquad\forall z\in S,
\end{align}
for some constant $L>0$ and the vector $\einsvec:=(1,1,\ldots,1)^\tT\in\R^n$. Then, for all $\mu_0,\mu_1\in\measm$, one has
\begin{align*}
\mom{\mu_{0}-S^\ell}&\le e^L\left(\mom{\mu_{1}-S^\ell}+ \W_\M(\mu_0,\mu_1)^2\right).
\end{align*}
\end{prop}

\begin{proof}
Since the assertion is trivial otherwise, assume that $\W_\M(\mu_0,\mu_1)<\infty$ and $\ell_2(\mu_{1,j}-S^\ell_j)<\infty$ for all $j=1,\ldots,n$.
Given $R>0$, let $\theta_R\in C^\infty_c(\R)$ with $\theta_R=\id$ on $[-R,R]$, $\theta_R=0$ on $\R\setminus [-3R,3R]$ and $|\theta_R'(x)|\le 1$ for all $x\in\R$. Observe that $\theta_R^2$ increases to $x\mapsto x^2$ as $R\nearrow \infty$. Let $(\mu,w)\in \scrC_1(\mu_0\to\mu_1)$ be such that $\act(\mu_t,w_t)=\W_\M(\mu_0,\mu_1)^2$ for almost all $t\in [0,1]$, by Proposition \ref{prop:minWM}. Let $s\in [0,1]$ be arbitrary. We first obtain that
\begin{align*}
\int_\R \theta_R^2 \einsvec^\tT(\mu_s-S^\ell)\dd x-\int_\R \theta_R^2 \einsvec^\tT(\mu_0-S^\ell)\dd x=-\int_0^s\int_\R \theta_R^2\einsvec^\tT \partial_x w_t\dd x \dd t.
\end{align*}
Using condition (C1), which yields the existence of a unique symmetric, positive definite square root $\M(z)^{1/2}$ of $\M(z)$, we get
\begin{align*}
&-\int_0^s\int_\R \theta_R^2\einsvec^\tT \partial_x w_t\dd x \dd t =\int_0^s\int_\R 2\theta_R\theta_R'\einsvec^\tT \M(\mu_t)^{1/2}\M(\mu_t)^{-1/2}w_t \dd x\dd t\\
&\le \int_0^s \int_\R (\theta_R\theta_R')^2\einsvec^\tT \M(\mu_t)\einsvec \dd x\dd t+\E_1(\mu,w),
\end{align*}
the last step being a consequence of the Cauchy-Schwarz and Young inequalities. Using the Lipschitz-type condition \eqref{eq:lipmob} and the bound on $\theta_R'$, we end up with
\begin{align*}
\int_\R \theta_R^2 \einsvec^\tT(\mu_s-S^\ell)\dd x-\int_\R \theta_R^2 \einsvec^\tT(\mu_0-S^\ell)\dd x&\le L\int_0^s\int_\R \theta_R^2\einsvec^\tT (\mu_t-S^\ell)\dd x \dd t+\W_\M(\mu_0,\mu_1)^2.
\end{align*}
Hence, by Gronwall's lemma,
\begin{align*}
\int_\R \theta_R^2\einsvec^\tT(\mu_s-S^\ell)\dd x &\le e^{Ls}\left(\W_\M^2(\mu_0,\mu_1)+\int_\R \theta_R^2\einsvec^\tT(\mu_0-S^\ell)\dd x\right),
\end{align*}
from which the assertion follows by monotone convergence $R\nearrow \infty$ for $s=1$.
\end{proof}

\subsection{Densities at finite distance}\label{subsec:finito}
In this section, we derive sufficient conditions under which $\W_\M(\mu_0,\mu_1)$ is finite. Throughout this section, the state space shall be a $n$-cuboid $S=[S^\ell,S^r]$.

\begin{prop}[Bounds on $\W_\M$ in terms of $\W_2$]
\label{prop:WM2}
Let a mobility $\M$ be given and assume that there exists a fully decoupled mobility $\M_0$ as in \eqref{eq:decoupled}, where the scalar mobilities $\mob_j$ are \emph{uniformly concave}, $\mob_j''\le -\delta$, 
for some $\delta>0$, and such that the following condition holds:
\begin{align}
\label{eq:mobdiagmob}
\exists K>0:\qquad A_K(z):=K\M_0(z)^{-1}-\M(z)^{-1}\in\Matn \text{ is positive definite.}
\end{align}
Let $\mu_0,\mu_1\in\measm$ with
\begin{align*}
\int_\R(\mu_0-S^\ell)\dd x=m=\int_\R(\mu_1-S^\ell)\dd x
\end{align*}
for some $m\in [0,\infty)^n$ and $\mom{\mu_0-S^\ell},\mom{\mu_1-S^\ell}<\infty$. Then the following statements hold:
\begin{enumerate}[(a)]
\item $\W_\M(\mu_0,\mu_1)$ is finite; in particular, one has
\begin{align}
\label{eq:WMell2}
\W_\M^2(\mu_0,\mu_1)&\le C[\mom{\mu_0-S^\ell}+\mom{\mu_1-S^\ell}]
\end{align}
with a constant $C>0$ depending on $m$.
\item If, moreover, for almost every $x\in\R$, one has $\mu_0(x),\,\mu_1(x)\le \tilde S^r$ for $\tilde {S^r}\in \inn{S}$, then
\begin{align}
\label{eq:WMW2}
\W_\M^2(\mu_0,\mu_1)&\le \tilde C\sum_{j=1}^n\W_2^2(\mu_{0,j}-S^\ell_j,\mu_{1,j}-S^\ell_j),
\end{align}
with a constant $\tilde C>0$ depending on $m$ and $\tilde {S^r}$.
\end{enumerate}
\end{prop}

\begin{proof}
For every $(\mu,w)\in\scrC_1(\mu_0\to\mu_1)$, one has due to condition \eqref{eq:mobdiagmob} that
\begin{align*}
\W_\M^2(\mu_0,\mu_1)&\le \E_1(\mu,w)\le K\sum_{j=1}^n\int_0^1\int_\R\frac{w_j^2}{\mob_j(\mu_j)}\dd x\dd t.
\end{align*}
Moreover, since the $\mob_j$ are uniformly concave, we have
\begin{align*}
\mob_j(\mu_j)&\ge \frac{\delta}{4}(\mu_j-S^\ell_j)(S^r_j-\mu_j)=:\tilde{\mob}_j(\mu_j), 
\end{align*}
and hence
\begin{align*}
\W_\M^2(\mu_0,\mu_1)&\le K\sum_{j=1}^n\int_0^1\int_\R\frac{w_j^2}{\tilde\mob_j(\mu_j)}\dd x\dd t.
\end{align*}
This estimate entails us to consider each component separately, by the same procedure as in the proof of \cite[Thm. 3]{lisini2010}.
\end{proof}

In the framework of perturbations of fully decoupled mobilities (cf. Section \ref{subsec:pdec}) for $n=2$ components, we are able to give a sufficient condition such that \eqref{eq:mobdiagmob} is true.

\begin{prop}[Estimate on $\M^{-1}$ for two components]\label{prop:estWM2}
Assume that, for small $\eps>0$, the mobility $\M$ is of the form
\begin{align*}
\M(z)&=\M_{0}(z)+\eps \M_\eps(z),\quad\text{where}\quad
\M_0(z):=\begin{pmatrix}\mob_1(z_1)&0\\0&\mob_2(z_2)\end{pmatrix},
\end{align*}
with a fully decoupled mobility $\M_0$. Assume that, in addition to (C0)--(C2), the following conditions are satisfied for some $C>0$:
\begin{enumerate}[({C3'}a)]
\item $\displaystyle{\frac{|\M_{\eps,11}(z)|}{\mob_1(z_1)}< C},$
\item $\displaystyle{\frac{|\M_{\eps,22}(z)|}{\mob_2(z_2)}< C},$
\item $\displaystyle{\frac{\mob_1(z_1)\mob_2(z_2)}{\det \M(z)}<C}.$
\end{enumerate}
Then, condition \eqref{eq:mobdiagmob} in Proposition \ref{prop:WM2} holds.
\end{prop}

\begin{proof}
We use the $\tr$-$\det$ criterion on $A_K(z)$ and have that (omitting the argument for the sake of clarity)
\begin{align}
\tr(A_K)>0\,&\Leftrightarrow\,K>\frac{\mob_1 \mob_2}{\det \M}\left(1+\eps\frac{\M_{\eps,11}+\M_{\eps,22}}{\mob_1+\mob_2}\right)>0,\label{eq:T}\\
\det(A_K)>0\,&\Leftrightarrow\,K^2-K\frac{2\mob_1\mob_2+\eps \M_{\eps,11}(z)\mob_2+\eps \M_{\eps,22}\mob_1)}{\det \M}+\frac{\mob_1\mob_2}{\det \M}>0\label{eq:D0}.
\end{align}
Using the assumptions on $\M$, one easily verifies that \eqref{eq:D0} holds if
\begin{align}
K>\frac{2\mob_1\mob_2+\eps \M_{\eps,11}\mob_2+\eps \M_{\eps,22}\mob_1}{\det \M}>0.\label{eq:D}
\end{align}
The middle terms in \eqref{eq:T} and \eqref{eq:D} are strictly bounded from above by $C(1+\eps C)$, where $C$ is the constant in (C3'a)--(C3'c). Hence, choosing $K:=C(1+\eps C)$ yields the assertion.
\end{proof}


\section{Geodesic convexity and gradient flows}\label{sec:geodconv}
In this part of the paper, we formally establish conditions on $\lambda$-geodesic convexity of entropy functionals $\ent$ 
appearing in \eqref{eq:pdesystem} with respect to the distance $\W_\M$. 
In advance of our main results, 
we introduce our method of proof by referring to abstract results in the literature adapted to the situation at hand.

\subsection{Preliminaries}\label{subsec:geo_pre}
We first briefly recall the abstract setting developed in \cite{liero2012,mielke2011}, which is a variant of the famous ``Otto calculus''.
The goal is to give the metric space $(\measm,\W_\M)$ a partial Riemannian structure.

A function $\mu\in\measm$ is called \emph{regular}, if $\mu$ is smooth and attains values in $\inn{S}$ only.
Clearly, regular functions lie dense in $\measm$.
At a regular $\mu$, we can interprete variations $v\in C^\infty_c(\R;\R^n)$ as tangent vectors to $\measm$ at $\mu$:
each such $v$ is associated to the curve $s\mapsto\mu+s v$ in $\measm$.
In the same spirit, each $\xi\in C^\infty(\R;\R^n)$ defines a linear functional on tangent vectors $v$
by means of the usual pairing in $L^2(\R;\R^n)$:
\begin{align*}
  \dual\xi v = \int_\R \xi(x)^\tT v(x)\dd x.
\end{align*}
Thanks to the metric structure of $(\measm,\W_\M)$, 
there exists a distinguished injective map of linear functionals to tangent vectors at regular points $\mu$, 
which is called the \emph{Onsager operator} $\K$ for the distance $\W_\M$:
\begin{align}
  \label{eq:onsager}
  \K(\mu)\xi = -\partial_x(\M(\mu)\partial_x\xi).
\end{align}
With these notions,
we write \eqref{eq:pdesystem} as an \emph{abstract evolution equation},
\begin{align}
  \label{eq:abstract_pde}
  \partial_t \mu =-\F(\mu),
\end{align}
with the nonlinear operator $\F:\measm\to\tang\measm$ given by
\begin{align}
  \label{eq:rhs}
  \F(\mu) :=-\partial_x(\M(\mu)\partial_x\ent'(\mu)) = \K(\mu)\ent'(\mu).
\end{align}
In the framework of \cite{liero2012,mielke2011},
the verification of $\lambda$-geodesic convexity of $\ent$ with respect to the distance $\W_\M$
is based on the \emph{Eulerian calculus} that has originally been developed in \cite{otto2005}; see also \cite{daneri2008}. 
Theorem \ref{thm:mielke} below summarizes the main result of that theory.

We remark that certain hypotheses are implicitly imposed in order to justify the calculations that lead to that result.
The main one is that there is a dense subset $\mathscr{M}_0\subset\measm$ of regular functions such that 
\eqref{eq:abstract_pde} possesses a smooth classical solution for each initial condition from $\mathscr{M}_0$,
and the associated flow maps $\flow^t:\,\mathscr{M}_0\to\measm$ are continuous in the topology of $(\measm,\W_\M)$,
for each time $t\ge0$.
It is then one of the consequences of Theorem \ref{thm:mielke} that $\flow^{(\cdot)}$ actually extends 
in a unique way to a continuous flow on all of $\measm$.
Further, one needs to assume that 
the underlying entropy functional $\ent:\,\measm\to\R\cup\{\infty\}$ is proper, lower semicontinuous and bounded below.

The abstract criterion for $\lambda$-convexity is the following.
\begin{thm}[Condition for convexity {\cite[Thm. 3.6]{liero2012}}]
  \label{thm:mielke}
  Let $\lambda\in\R$ and let $\ent$, $\F$ and $\K$ be defined as in \eqref{eq:onsager}\&\eqref{eq:rhs}. 
  If
  \begin{align}
    \label{eq:mcond}
    \dual{\xi}{\dff\F(\mu)\K(\mu)\xi}-\frac{1}{2}\dual{\xi}{\dff\K(\mu)[\F(\mu)]\xi}&\ge \lambda\dual{\xi}{\K(\mu)\xi}
  \end{align}
  holds for all regular $\mu\in\measm$ and $\xi\in C^\infty(\R;\R^n)$ with $\partial_x \xi$ of compact support, 
  then $\flow^{(\cdot)}$ satisfies the evolution variational estimate \eqref{eq:evi} for $\ent$ 
  and hence defines a $\lambda$-flow on $(\measm,\W_\M)$.
  Further, $\ent$ is $\lambda$-geodesically convex w.r.t. $\W_\M$.
\end{thm}

\subsection{The multi-component heat equation}\label{subsec:heat}
In this section, we apply the theory of Section \ref{subsec:geo_pre} to the case of the \emph{multi-component heat equation},
\begin{align}
  \label{eq:multiheat}
  \partial_t\mu = \partial_{xx}\mu,
\end{align}
which is \eqref{eq:abstract_pde} for $\F(\mu)=-\partial_{xx}\mu$.
In this case, the flow maps $\flow^t:\measm\to\measm$ are explicitly known: 
\begin{align*}
  \flow^t(\mu^0) = \krnl_t\ast\mu^0,
\end{align*}
with the heat kernel $\krnl$ from \eqref{eq:heatkern}, for each $t>0$ and arbitrary initial data $\mu^0\in\measm$. Moreover, if $\mu^0$ is a smooth function with values in $\inn{S}$ only, 
then it follows by classical results that $(t,x)\mapsto\big(\flow^t(\mu^0)\big)(x)$ is also smooth on $[0,\infty)\times\R$,
and attains values in $\inn{S}$ only.
We are thus in the framework described above and conclude the following with the help of Theorem \ref{thm:mielke}.

\begin{prop}[The heat flow as a gradient flow]
  \label{prop:heat_conv}
  Assume that $\M:\,S\to\Matn$ satisfies (C0)--(C3), 
  and that $\M$ is induced by $h$ as in \eqref{eq:Minduce}, i.e.  $\M(z)=(\hess h(z))^{-1}$ at every $z\in\inn{S}$, for a continuous function $h:\,S\to\R$ which is smooth on $\inn{S}$. Suppose that for each $\mu_0,\mu_1\in\measm$ with $\W_\M(\mu_0,\mu_1)<\infty$, condition \eqref{eq:addsmooth} is satisfied.
  Then the flow map $\flow^{(\cdot)}$ for \eqref{eq:multiheat} defined above is a $0$-flow on $\measm$,
  and it is the gradient flow of the functional $\calH(\mu):=\int_\R h(\mu)\dd x$, 
  which is $0$-geodesically convex w.r.t. $\W_\M$.
\end{prop}
\begin{proof}
  To begin with, observe that with $\calH$ defined as above,
  \begin{align*}
    \M(\mu)\partial_x\calH'(\mu)=\M(\mu)\hess h(\mu)\partial_x\mu=\Id \partial_x\mu,
  \end{align*}
  which means that \eqref{eq:abstract_pde} simplifies to \eqref{eq:multiheat}.
  We verify \eqref{eq:mcond} for $\lambda=0$:
  for a given smooth $w:\R\to\R^n$, the relevant derivative expressions amount to
  \begin{align}
    \dff\F(\mu)[w] &=-\partial_{xx}w,\nonumber\\
    \dff\K(\mu)[w]\xi &=-\partial_x(\dff\M(\mu)[w]\partial_x \xi).\label{eq:DK}
  \end{align}
  We substitute this into the left-hand side of \eqref{eq:mcond} and integrate by parts to obtain
  \begin{align*}
    \dual{\xi}{\dff\F(\mu)[\K(\mu)\xi]}-\frac{1}{2}\dual{\xi}{\dff\K(\mu)[\F(\mu)]\xi}
    &=\dual{\xi}{\partial_{xxx}(\M(\mu)\partial_x\xi)}-\frac12\dual{\xi}{\partial_x(\dff\M(\mu)[\partial_{xx}\mu]\partial_x\xi)} \\
    &=-\frac12\dual{\partial_x\xi}{\dff^2\M(\mu)[\partial_x\mu,\partial_x\mu]\partial_x\xi}+\dual{\partial_{xx}\xi}{\M(\mu)\partial_{xx}\xi},
  \end{align*}
  which is nonnegative because of (C1) and (C2).
\end{proof}

\subsection{Internal energy functionals}
\label{subsec:funct}
We now study geodesic convexity of more general functionals of the form
\begin{align}
  \label{eq:funct_state}
  \ent(\mu)=\int_\R f(\mu(x))\dd x,
\end{align}
with a smooth function $f:\,\inn{S}\to\R$.
For brevity, we call these functionals \emph{internal energies}, 
regardless of their actual interpretation in physics or other sciences.
Our main result is Proposition \ref{prop:mccann} below, 
which is a further generalization of the generalized McCann condition 
established by Carrillo \emph{et al.} \cite{carrillo2010} for \emph{scalar} nonlinear mobilities ($n=1$).

\subsubsection{A generalized McCann condition}\label{ssubsec:mccann}
The main result of this section is the following sufficient criterion for $0$-contractivity
of the flow generated by the evolution equation
\begin{align}
  \label{eq:pmek}
  \partial_t\mu = \partial_x\big(\bL(\mu)\partial_x\mu\big),
  \quad\text{with} \quad \bL(z) = \M(z)\hess f(z),
\end{align}
which is \eqref{eq:pdesystem} for $\ent$ from \eqref{eq:funct_state},
i.e., the formal gradient flow of $\ent$ in $\W_\M$.
\begin{prop}[Multi-component McCann condition]\label{prop:mccann}
  Given a mobility matrix $\M$ that satisfies $\mathrm{(C0)}$--$\mathrm{(C2)}$ 
  and a functional $\ent$ of the form \eqref{eq:funct_state}, 
  assume that for all $z\in\inn{S}$ and all $v,\zeta,\beta\in\R^n$
  (omitting the argument $z$ from $\M=\M(z)$ and from $\bL=\M(z)\hess f(z)$):
  \begin{align}
    \label{eq:genmccann}
    \begin{split}
      0 \le &-\frac12v^\tT\,\dff^2\M[\zeta,\bL \zeta]\,v + \beta^\tT\,\bL\M\,\beta \\
      &+ \beta^\tT\big(\bL\dff\M[\zeta] - \dff\M[\bL \zeta]\big)v
      + v^\tT\,\dff\bL[\zeta]\big(\dff\M[\zeta]v+\M \beta\big) - v^\tT\,\dff\bL\big[\dff\M[\zeta]v+\M \beta\big]\,\zeta.
    \end{split}
  \end{align}
  Then, under the assumption of sufficient regularity of the associated flow generated by \eqref{eq:pmek}, 
  the functional $\ent$ is $0$-geodesically convex w.r.t. the distance $\W_\M$.
\end{prop}

\begin{proof}
  This is another application of Theorem \ref{thm:mielke}. 
  Let therefore $\mu\in\measm$ be regular and $\xi,w\in C^\infty(\R;\R^n)$, $\partial_x\xi$ with compact support.
  Observe that
  \begin{align*}
    \F(\mu) &=-\partial_x(\bL(\mu)\mu_x),\\
    \dff \F(\mu)[w] &=-\partial_x(\dff\bL(\mu)[w]\mu_x)-\partial_x(\bL(\mu)w_x),
  \end{align*}
  and in addition, \eqref{eq:DK} holds.
 Hence, integrating by parts, we obtain
\begin{align*}
&\dual{\xi}{\dff\F(\mu)\K(\mu)\xi}-\frac12\dual{\xi}{\dff\K(\mu)[\F(\mu)]\xi}\\
&=-\dual{\xi_x}{\dff\bL(\mu)[\partial_x(\M(\mu)\xi_x)]\mu_x}+\dual{\xi_{xx}}{\bL(\mu)\partial_x(\M(\mu)\xi_x)}+\dual{\xi_{xx}}{\dff\bL(\mu)[\mu_x]\partial_x(\M(\mu)\xi_x)}&\\
&-\frac12\dual{\xi_{xx}}{\dff\M(\mu)[\bL(\mu)\mu_x]\xi_x}-\frac12\dual{\xi_x}{\dff\M(\mu)[\bL(\mu)\mu_x]\xi_{xx}}-\frac12\dual{\xi_x}{\dff^2\M(\mu)[\mu_x,\bL(\mu)\mu_x]\xi_x}\\
&=-\frac12\dual{\xi_x}{\dff^2\M(\mu)[\mu_x,\bL(\mu)\mu_x]\xi_x}-\dual{\xi_{xx}}{\dff\M(\mu)[\bL(\mu)\mu_x]\xi_x}-\dual{\xi_x}{\dff\bL(\mu)[\dff\M(\mu)[\mu_x]\xi_x+\M(\mu)\xi_{xx}]\mu_x}\\
&+\dual{\xi_{xx}}{\bL(\mu)(\dff\M(\mu)[\mu_x]\xi_x+\M(\mu)\xi_{xx})}+\dual{\xi_{xx}}{\dff\bL(\mu)[\mu_x](\dff\M(\mu)[\mu_x]\xi_x+\M(\mu)\xi_{xx})}.
\end{align*}
Condition \eqref{eq:genmccann} now implies pointwise nonnegativity (substitute $v:=\xi_x(x)$, $\beta:=\xi_{xx}(x)$, $\zeta:=\mu_x(x)$ for $x\in\R$) and consequently \eqref{eq:mcond} for $\lambda=0$.
\end{proof}

\begin{remark}[Diagonal mobility]\label{rem:mccanndiag}
In the case of a fully decoupled mobility matrix 
\begin{align*}
\M(z)&=\begin{pmatrix}\mob_1(z_1)& 0\\ 0 & \mob_2(z_2)\end{pmatrix}
\end{align*}
for $n=2$ components, where in general $\frac12 \mob_j''\mob_j+(\mob_j')^2\neq 0$, the generalized McCann condition \eqref{eq:genmccann} is equivalent to
\begin{align*}
\partial_{11}f(z)&\ge 0,\quad\partial_{22}f(z)\ge 0,\quad\partial_{12}f(z)=0,
\end{align*}
since \eqref{eq:genmccann} reads in this case 
\begin{align*}
0&\ge\left[\frac12 v_1^2\zeta_1^2\mob_1''\mob_1-\beta_1^2\mob_1^2\right]\partial_{11}f+\left[\frac12 v_2^2\zeta_2^2\mob_2''\mob_2-\beta_2^2\mob_2^2\right]\partial_{22}f\\
&+\left[\frac12 v_1^2\zeta_1\zeta_2\mob_1''\mob_1+\frac12 v_2^2\zeta_1\zeta_2\mob_2''\mob_2-2\beta_1\beta_2\mob_1\mob_2-2v_1v_2\zeta_1\zeta_2\mob_1'\mob_2'\right]\partial_{12}f\\
&+\left[[v_1^2(\mob_1')^2+v_2^2(\mob_2')^2]\zeta_1\zeta_2+2v_1\beta_1\zeta_2\mob_1'\mob_1+2v_2\beta_2\zeta_1\mob_2'\mob_2-2v_2\beta_1\zeta_2\mob_2'\mob_1-2v_1\beta_2\zeta_1\mob_1'\mob_2\right]\partial_{12}f.
\end{align*}
Imposing e.g. $\beta=0$, $v_1=1$, $v_2=0$ and $\zeta_1=1$, one obtains 
\begin{align*}
0&\ge \frac12 \mob_1''\mob_1\partial_{11}f+\left[\frac12 \mob_1''\mob_1+(\mob_1')^2\right]\zeta_2\partial_{12}f,
\end{align*}
from which necessarily $\partial_{12}f(\mu)=0$ follows.
Hence, the only possible choice is $f(z):=\psi_1(z_1)+\psi_2(z_2)$ with convex functions $\psi_1,\,\psi_2$. We solely recover the generalized McCann condition for $n=1$ (cf. \cite{carrillo2010}) for each of the two components separately if $\M$ is fully decoupled.
\end{remark}

\subsubsection{Perturbation results and examples}\label{ssubsec:mccann_x}
This paragraph is devoted to examples satisfying condition \eqref{eq:genmccann} of Proposition \ref{prop:mccann}. In particular, we investigate suitable perturbations of the entropies having the heat flow as gradient flow, cf. Proposition \ref{prop:heat_conv}. We first start with a more general result involving perturbations of compact support in $\inn{S}$ and continue with a specific example where the support of the perturbation extends to all of $S$.
\begin{prop}[Perturbations of compact support]
Let a mobility $\M$ satisfy the conditions (C0)--(C3) and the stronger condition (C2') and be induced by $h$ as in \eqref{eq:Minduce}. For $\alpha,\tilde\eps>0$ and $g\in C^\infty_c(\inn{S})$, define $f(z):=\alpha h(z)+\tilde\eps g(z)$ and $\ent$ according to \eqref{eq:funct_state}. Then, for $\tilde\eps>0$ sufficiently small, the generalized McCann condition \eqref{eq:genmccann} is satisfied.
\end{prop}

\begin{proof}
If $z\notin\supp g$, the conditions (C1) and (C2') directly yield the claim. Furthermore, there exists a constant $\delta_g>0$ such that for all $z\in\supp g$, one has
\begin{align*}
\beta^\tT\dff^2 \M(z)[\zeta,\zeta]\beta&\le -\delta_g |\beta|^2|\zeta|^2,\\
-\gamma^\tT \M(z)\gamma &\le -\delta_g |\gamma|^2,
\end{align*}
for all $\beta,\gamma,\zeta\in\R^n$.
Hence, by continuity, we obtain for the r.h.s. in \eqref{eq:genmccann}, recalling
\begin{align*}
\bL(z)&=\M(z)\hess f(z)=\alpha\Id+\tilde\eps \hess g(z):
\end{align*}
\begin{align*}
&-\frac12v^\tT\,\dff^2\M[\zeta,\bL \zeta]\,v + \beta^\tT\,\bL\M\,\beta \\
      &+ \beta^\tT\big(\bL\dff\M[\zeta] - \dff\M[\bL \zeta]\big)v
      + v^\tT\,\dff\bL[\zeta]\big(\dff\M[\zeta]v+\M \beta\big) - v^\tT\,\dff\bL\big[\dff\M[\zeta]v+\M \beta\big]\,\zeta\\
&\ge \frac{\alpha}{2}\delta_g|\zeta|^2|v|^2+\alpha\delta_g|\beta|^2-\tilde\eps C_{g,\M}(|\zeta|^2|v|^2+|\beta|^2+|\zeta||v||\beta|),
\end{align*}
with a constant $C_{g,\M}>0$. Using Young's inequality, one immediately deduces that the r.h.s. is nonnegative and thus \eqref{eq:genmccann} is satisfied, provided that $\tilde\eps\le \frac{\alpha\delta_g}{3C_{g,\M}}$.
\end{proof}

We conclude this section with a specific example such that the support of the perturbation $g$ extends to all of $S$.
\begin{xmp}[Non-compactly supported perturbations]
Let $\M$ be induced by $h$ from \eqref{eq:hnull}\&\eqref{eq:heps}:
\begin{align*}
h(z)&:=z_1\log(z_1)+(1-z_1)\log(1-z_1)+z_2\log(z_2)+(1-z_2)\log(1-z_2)+\eps d_1d_2,\\
d_j&:=z_j(1-z_j),
\end{align*}
and $\eps>0$ chosen so small such that the conditions (C0)--(C3) and (C2') are satisfied. Define furthermore $\tilde g:[0,\frac14]^2\to \R$ by
\begin{align*}
\tilde g(m_1,m_2):=\exp\left(-\frac1{m_1}-\frac1{m_2}\right)
\end{align*}
for all $0<m_1,m_2\le \frac14$, and $\tilde g(m_1,0)=0=\tilde g(0,m_2)$. Consider now for $\tilde\eps>0$ the map $f(z):=h(z)+\tilde\eps \tilde g(d_1,d_2)$ and the functional $\ent$ according to \eqref{eq:funct_state}. Then, for $\tilde\eps>0$ sufficiently small, the generalized McCann condition \eqref{eq:genmccann} is satisfied.
\end{xmp}

Our idea of proof relies on the structure of $\M$ in this particular case (cf. Section \ref{subsec:pdec}): There exists a positive rational function $r_1:\,\left(0,\frac14\right]^2\to (0,\infty)$ with $\lim\limits_{\tilde m\to 0}r_1(m_1,\tilde m)=0=\lim\limits_{\tilde m\to 0}r_1(\tilde m,m_2)$ for all $(m_1,m_2)\in \left(0,\frac14\right]^2$, such that 
\begin{align*}
\frac12 \beta^\tT\dff^2\M(z)[\zeta,\zeta]\beta-\gamma^\tT\M(z)\gamma&\le -r_1(d_1,d_2)(|\zeta|^2|\beta|^2+|\gamma|^2).
\end{align*}
Furthermore, there exists another rational function $r_2:\,\left(0,\frac14 \right]^2\to [0,\infty)$ such that the following estimate on the r.h.s. in condition \eqref{eq:genmccann} is possible:
\begin{align*}
&-\frac12v^\tT\,\dff^2\M[\zeta,\bL \zeta]\,v + \beta^\tT\,\bL\M\,\beta \\
      &+ \beta^\tT\big(\bL\dff\M[\zeta] - \dff\M[\bL \zeta]\big)v
      + v^\tT\,\dff\bL[\zeta]\big(\dff\M[\zeta]v+\M \beta\big) - v^\tT\,\dff\bL\big[\dff\M[\zeta]v+\M \beta\big]\,\zeta\\
&\ge (r_1(d_1,d_2)-\tilde g(d_1,d_2)r_2(d_1,d_2))(|\zeta|^2|\beta|^2+|\gamma|^2).
\end{align*}

Since for all $(m_1,m_2)\in \left(0,\frac14\right]^2$, one has
\begin{align*}
\lim_{\tilde m\to 0}\tilde g(\tilde m,m_2)\frac{r_2(\tilde m,m_2)}{r_1(\tilde m,m_2)}=0=\lim_{\tilde m\to 0}\tilde g(m_1,\tilde m)\frac{r_2(m_1,\tilde m)}{r_1(m_1,\tilde m)},
\end{align*}
we find $\tilde\eps_0>0$ sufficiently small such that \eqref{eq:genmccann} holds for all $0<\tilde\eps\le \tilde\eps_0$.

\subsection{The potential energy}
\label{subsec:xfunct}
In this section, we study $\lambda$-convexity of the regularized \emph{potential energy} functional
\begin{align}
  \label{eq:funct_space}
  \pot(\mu)=\int_\R\left[\alpha h(\mu)+\rho(x)^\tT\mu\right]\dd x,
\end{align} 
which has a density depending explicitly on the spatial variable $x$. Here, $\W_\M$ and $h$ are as in Proposition \ref{prop:heat_conv} and $\alpha>0$, $\rho\in C^\infty_c(\R;\R^n)$ are fixed. The flow associated to $\pot$ is generated by the following (regularized) nonlinear \emph{transport equation}:
\begin{align}
\label{eq:transport}
\partial_t \mu&=\alpha\partial_{xx}\mu+\partial_x(\M(\mu)\partial_x\rho).
\end{align}

\subsubsection{Convexity}
\label{ssubsec:potential}
A sufficient condition on convexity of those entropies is the following:
\begin{prop}[Convexity for the regularized potential energy functional]
  \label{prop:pot_ent}
  Let $\pot$ be of the form \eqref{eq:funct_space} with $h$, $\alpha$ and $\rho$ as mentioned above, 
  let $\M=(\hess h)^{-1}$ be as in Proposition \ref{prop:heat_conv} and $\lambda\in\R$ be fixed. 
  If for all $z\in\inn{S}$ and all $v,\zeta\in \R^n$, $q^1,q^2\in\overline{\ball_{R}(0)}$, $R:=\|\rho\|_{C^2}$, the condition
  \begin{align}
    \label{eq:genconv}
\begin{split}
    0&\le - \frac{\alpha}{2}v^\tT\dff^2\M[\zeta,\zeta]v - \lambda v^\tT\M v\\
    &-\frac{1}{2}v^\tT \dff^2\M[\zeta,\M q^1]v + v^\tT \dff^2\M[\zeta,\M v]q^1 + v^\tT \dff\M[\M q^2]v
\end{split}
  \end{align}
  is satisfied, then $\pot$ is $\lambda$-geodesically convex w.r.t. the distance $\W_\M$ 
  under the assumption of sufficient regularity of the associated flow generated by \eqref{eq:transport}.
\end{prop}
\begin{proof}
The method of proof is similar to that of Proposition \ref{prop:mccann}. Here, one gets
\begin{align*}
\F(\mu)&=-\alpha\partial_{xx}\mu-\partial_x(\M(\mu)\rho_x),\\
\dff \F(\mu)[w]&=-\alpha \partial_{xx}w-\partial_x(\dff\M(\mu)[w]\rho_x).
\end{align*}
Consequently, performing essentially the same calculations as in the proofs of the Propositions \ref{prop:heat_conv} and \ref{prop:mccann},
\begin{align*}
&-\frac12\dual{\xi}{\dff\K(\mu)[\F(\mu)]\xi}+\dual{\xi}{\dff\F(\mu)\K(\mu)\xi}-\lambda\dual{\xi}{\K(\mu)\xi}\\
&=-\frac{\alpha}{2}\dual{\xi_x}{\dff^2\M(\mu)[\mu_x,\mu_x]\xi_x}-\frac12\dual{\xi_x}{\dff^2\M(\mu)[\mu_x,\M(\mu)\rho_x]\xi_x}\\
&+\alpha \dual{\xi_{xx}}{\M(\mu)\xi_{xx}}+\dual{\xi_x}{\dff^2\M(\mu)[\mu_x,\M(\mu)\xi_x]\rho_x}+\dual{\xi_x}{\dff\M(\mu)[\M(\mu)\xi_x]\rho_{xx}}\\
&-\lambda \dual{\xi_x}{\M(\mu)\xi_x}.
\end{align*}
We use the fact that for all $\gamma,q^1,v\in\R^n$ and all $z\in\inn{S}$, one has due to symmetry of the third-order tensor $\dff^3 h$:
\begin{align*}
\gamma^\tT\dff\M(z)[\M(\mu) q^1]v&=-\dff^3 h(z)[\M(z)\gamma,\M(z)q^1,\M(z)v]=\gamma^\tT\dff\M(z)[\M(z) v]q^1.
\end{align*}
Hence, we obtain
\begin{align*}
&\dual{\xi}{\dff\F(\mu)\K(\mu)\xi}-\frac12\dual{\xi}{\dff\K(\mu)[\F(\mu)]\xi}-\lambda\dual{\xi}{\K(\mu)\xi}\\
&=-\frac{\alpha}{2}\dual{\xi_x}{\dff^2\M(\mu)[\mu_x,\mu_x]\xi_x}+\alpha \dual{\xi_{xx}}{\M(\mu)\xi_{xx}}-\lambda \dual{\xi_x}{\M(\mu)\xi_x}\\
&-\frac12\dual{\xi_x}{\dff^2\M(\mu)[\mu_x,\M(\mu)\rho_x]\xi_x}+\dual{\xi_x}{\dff^2\M(\mu)[\mu_x,\M(\mu)\xi_x]\rho_x}+\dual{\xi_x}{\dff\M(\mu)[\M(\mu)\rho_{xx}]\xi_x},
\end{align*}
which is nonnegative due to condition \eqref{eq:genconv} and (C1) (substitute $v:=\xi_x(x)$, $\zeta:=\mu_x(x)$, $q^1:=\rho_x(x)$, $q^2:=\rho_{xx}(x)$ for $x\in\R$) and hence implies \eqref{eq:mcond}.
\end{proof}

\subsubsection{The case of fully decoupled mobility}\label{ssubsec:decpot}
In this paragraph, we consider the case of a fully decoupled mobility (cf. Section \ref{subsec:decoup})
\begin{align*}
  \M(z) =
  \begin{pmatrix}
    \mob_1(z_1) & & \\ & \ddots & \\ & & \mob_n(z_n)
  \end{pmatrix},
\end{align*}
on the $n$-cuboid $S=[S^\ell,S^r]$. We shall assume that the scalar mobilities $\mob_j$ are such that
\begin{itemize}
\item $\mob_j\in C^2([S^\ell_{j},S^r_j])$,
\item $\mob_j(s)>0$ for $s\in (S^\ell_{j},S^r_j)$ and $\mob_j(S^\ell_j)=\mob_j(S^r_j)=0$,
\item $\mob_j''(s)\le 0$ for $s\in [S^\ell_{j},S^r_j]$.
\end{itemize}
Recall that $\M$ is of the special form \eqref{eq:Minduce} $\M(z)=(\hess h(z))^{-1}$, where 
\begin{align*}
h(z)=\sum_{j=1}^n h_j(z_j),
\end{align*}
$h_j$ being a second primitive of $\frac1{\mob_j}$.

\begin{prop}[$\lambda$-convexity of the potential energy]\label{prop:convpot_dec}
For a fully decoupled mobility $\M$ as mentioned above, fix $\alpha>0$ and $\rho\in C^\infty_c(\R;\R^n)$ and consider the regularized potential energy functional $\pot$ defined in \eqref{eq:funct_space}.
\begin{enumerate}[(a)]
\item Let $\zref\in S$ and $\mu^0\in\measm$ such that $\mu^0-\zref\in H^1(\R;\R^n)$ and such that $\mu^0$ attains values in $\inn{S}$ only. Then, the initial-value problem for \eqref{eq:transport} 
\begin{align}\label{eq:ivptrans}
\partial_t \mu&=\alpha\partial_{xx}\mu+\partial_x(\M(\mu)\partial_x\rho),\qquad \mu(0,\cdot)=\mu^0,
\end{align}
possesses a unique local-in-time classical solution $\mu:\,[0,T]\to\measm$ with $\mu-\zref\in C^0([0,T];H^1(\R;\R^n))$, where $T=T(\mu^0,\rho)>0$.
\item There exists $C=C(\rho)>0$ such that condition \eqref{eq:genconv} in Proposition \ref{prop:pot_ent} is satisfied for all $\lambda\le -C(\frac{1}{\alpha}+1)$.
\end{enumerate}
Hence, Proposition \ref{prop:pot_ent} is applicable and yields $\lambda$-convexity of the potential energy $\pot$.
\end{prop}

\begin{proof}
\begin{enumerate}[(a)]
\item See Appendix \ref{app:ivptrans}.
\item We proceed similarly to \cite{lisini2012} and observe that for all $z\in\inn{S}$ and all $v,\zeta\in \R^n$, $q^1,q^2\in\overline{\ball_{R}(0)}$, $R:=\|\rho\|_{C^2}$, one has
\begin{align*}
&- \frac{\alpha}{2}v^\tT\dff^2\M[\zeta,\zeta]v -\frac{1}{2}v^\tT \dff^2\M[\zeta,\M q^1]v + v^\tT \dff^2\M[\zeta,\M v]q^1 + v^\tT \dff\M[\M q^2]v\\
&=\sum_{j=1}^n\left[-\frac{\alpha}{2}\mob_j''(z_j)\zeta_j^2+\frac12 \mob_j''(z_j)\mob_j(z_j)q^1_j\zeta_j+\mob_j'(z_j)\mob_j(z_j)q_j^2\right]v_j^2\\
&\ge \sum_{j=1}^n\left[-\frac{1}{8\alpha}|\mob_j''(z_j)|\mob_j(z_j)^2(q^1_j)^2+\mob_j'(z_j)\mob_j(z_j)q_j^2\right]v_j^2,
\end{align*}
the last step being a consequence of Young's inequality. Using the bounds on $\mob_j$, $q^1$ and $q^2$, we obtain
\begin{align*}
\sum_{j=1}^n\left[-\frac{1}{8\alpha}|\mob_j''(z_j)|\mob_j(z_j)(q^1_j)^2+\mob_j'(z_j)q_j^2\right]\mob_j(z_j)v_j^2&\ge -\sum_{j=1}^n \|\mob_j\|_{C^2}R\left[\frac{\|\mob_j\|_{C^2}R}{8\alpha}+1\right]\mob_j(z_j)v_j^2.
\end{align*}
Obviously, for all $\lambda\le-\max\limits_j\|\mob_j\|_{C^2}R\left[\frac{\|\mob_j\|_{C^2}R}{8\alpha}+1\right]$, \eqref{eq:genconv} holds.
\end{enumerate}
\end{proof}


\section{Existence of weak solutions}\label{sec:weak}

In this section, we prove the existence of weak solutions for a class of initial-value problems of the form \eqref{eq:pdesystem}. More specifically, we consider the case of a fully decoupled mobility $\M$ but allow for coupling inside the driving entropy $\ent$. Note that, by Remark \ref{rem:mccanndiag}, the functional $\ent$ will in general \emph{not} be geodesically convex .

\subsection{Setting and basic properties}\label{subsec:sett}

We again consider as state space a $n$-cuboid $S=[S^\ell,S^r]\subset\R^n$ and let $h:S\to\R$, $h(z)=\sum_{j=1}^n h_j(z_j)$, where for all $j=1,\ldots,n$:
\begin{enumerate}[(H1)]
\setcounter{enumi}{-1}
\item $h_j$ is $\alpha$-Hölder continuous on $[S^\ell_j,S^r_j]$ for some $\alpha\in (0,1]$ and smooth on $(S^\ell_j,S^r_j)$,
\item $h_j$ is strictly convex,
\item $\lim\limits_{s\searrow S^\ell_j}h_j''(s)=+\infty=\lim\limits_{s\nearrow S^r_j}h_j''(s)$.
\item $\frac1{h_j''}$ is concave and can be extended at the boundary $\{S^\ell_j,S^r_j\}$ to a function in $C^2([S^\ell_j,S^r_j])$.
\end{enumerate}

Obviously, the induced fully decoupled mobility $\M$ as in Section \ref{ssubsec:decpot} satisfies the requirements of that section, in particular also (C0)--(C3), if $h$ satisfies (H0)--(H3).

Furthermore, let $\eta\in C^\infty_c(\R;\R^n)$ and $f:\,S\to\R$ such that
\begin{enumerate}[(A)]
\setcounter{enumi}{5}
\item $f$ is smooth and uniformly convex, i.e. $\hess f(z)\ge C_f\Id$ for all $z\in S$ and some $C_f>0$.
\end{enumerate}

We introduce a \emph{reference state} $\zref\in S$, i.e. a constant level relatively to which certain quantities (e.g. the mass of an element in $\measm$) will be measured. We distinguish two qualitatively different cases:
\begin{enumerate}[(A)]
\item Reference state $\zref=S^\ell$.
\item Reference state $\zref\in\inn{S}$. 
\end{enumerate}
The respective case will be indicated with (A) and/or (B) in definitions and statements. Note that in case (A), the function $\mu-\zref$ is nonnegative for each $\mu\in\measm$.\\

\begin{definition}[Heat and driving entropy]
Let $\zref,f,h,\eta$ be as mentioned above. In case (A), let $\mu^0\in\measm$ be such that $m:=\|\mu^0-\zref\|_{L^1}\in (0,\infty)$.
Define the \emph{heat entropy} functional by
\begin{align*}
\calH(\mu)=\int_\R h_{\zref}(\mu)\dd x,
\end{align*}
where 
\begin{enumerate}[(A)]
\item $h_{\zref}:=h(z)-h(\zref)$,
\item $h_{\zref}(z):=h(z)-h(\zref)-(z-\zref)^\tT\grd h(\zref)$.
\end{enumerate}
The \emph{driving entropy} functional $\ent:\measm\to\R\cup\{\infty\}$ is defined by
\begin{align*}
\ent(\mu)=\begin{cases}\int_\R[f(\mu)-f(\zref)-(\mu-\zref)^\tT\grd f(\zref)+\mu^\tT\eta]\dd x,&\text{if }\mu \in \Xaux, \\ +\infty,&\text{otherwise,}\end{cases}
\end{align*}
where 
\begin{enumerate}[(A)]
\item $\Xaux :=\{\mu\in\measm:\,\|\mu-\zref\|_{L^1}=m,\,\mom{\mu-\zref}<\infty\}$,
\item $\Xaux :=\{\mu\in\measm:\,\|\mu-\zref\|_{L^2}<\infty\}$.
\end{enumerate}
\end{definition}

Note that in both cases, $h_\zref(\zref)=0$ and $h_\zref$ is strictly convex with $\hess h_\zref=\hess h$. In case (B), $h_\zref$ is nonnegative.

\begin{xmp}\label{ex:HF}
\begin{enumerate}[(a)]
\item The paradigmatic example for $h$ satisfying (H0)--(H3) is given by
\begin{align*}
h_j(s)=\begin{cases}(s-S^\ell_j)\log(s-S^\ell_j)+(S^r_j-s)\log\left(S^r_j-s\right)-(S^r_j-S^\ell_j)\log(S^r_j-S^\ell_j),&\text{if }s\in(S^\ell_j,S^r_j),\\ 0,&\text{if }s\in \{S^\ell_j,S^r_j\},\end{cases}
\end{align*}
yielding
\begin{align*}
\mob_j(s)=\frac1{S^r_j-S^\ell_j}(s-S^\ell_j)(S^r_j-s).
\end{align*}
\item An admissible choice for $f$ is
\begin{align*}
f(z)=\frac12 z^\tT Q z+\eps r(z),
\end{align*}
where $Q\in\Matn$ is symmetric positive definite, $r:S\to\R$ is smooth and $\eps\ge 0$ is such that $Q+\eps \hess r(z)$ is positive definite for all $z\in S$.
\end{enumerate}
\end{xmp}

\begin{prop}[Properties of heat and driving entropy (A)+(B)]\label{prop:heatentropy}
The following statements hold:
\begin{enumerate}[(a)]
\item $\calH$ is finite on $\Xaux$.
\item For all $\mu_0,\mu_1\in \Xaux$ with $\W_\M(\mu_0,\mu_1)<\infty$, condition \eqref{eq:addsmooth} holds for $h_\zref$ in place of $h$.
\item The Lipschitz-type condition \eqref{eq:lipmob} holds.
\item There exist constants $\underline{C},\,\overline{C}>0$ such that for all $\mu\in\Xaux$, the following holds:
\begin{align*}
\underline{C}(\|\mu-\zref\|_{L^2}^2-1)&\le \ent(\mu)\le \overline{C}(\|\mu-\zref\|_{L^2}^2+1).
\end{align*}
In particular, $\ent$ is finite on $\Xaux$.
\item If $\mu_k-\zref\rightharpoonup \mu-\zref$ weakly in $L^2(\R;\R^n)$, then
\begin{align*}
\ent(\mu)\dd x&\le \liminf_{k\to\infty} \ent(\mu_k)\dd x.
\end{align*}
\end{enumerate}
\end{prop}

\begin{proof}
\begin{enumerate}[(a)]
\item We distinguish both cases.
\begin{enumerate}[(A)]
\item Due to $\alpha$-Hölder continuity of $h$, there exists $C>0$ such that for all $z\in S$:
\begin{align*}
|h_\zref(z)|&\le C\sum_{j=1}^n|z_j-\zref_j|^\alpha.
\end{align*}
By Hölder's inequality, we then deduce for $\mu\in\Xaux$:
\begin{align*}
|\calH(\mu)|\le C\sum_{j=1}^n\int_\R |\mu_j-\zref_j|^\alpha\dd x\le C\sum_{j=1}^n \left(\int_\R(\mu_j-\zref_j)(x^2+1)\dd x\right)^\alpha\left(\int_\R (x^2+1)^{\frac{\alpha}{\alpha-1}}\dd x\right)^{1-\alpha},
\end{align*}
which is finite thanks to the definition of $\Xaux$.
\item Obviously, since $h$ is smooth in a neighbourhood of $\zref$ and bounded on the whole of $S$, there exists $C>0$ such that for all $z\in S$:
\begin{align*}
h_\zref(z)&\le C|z-\zref|^2,
\end{align*}
which proves the claim.
\end{enumerate}
\item
\begin{enumerate}[(A)]
\item Thanks to the properties of the heat kernel, $\krnl_\delta\ast\mu\in\Xaux$ if $\mu\in\Xaux$ since mass is conserved and the second moment grows linearly in time along the heat flow. Hence, by part (a), $\calH(\mu_i)-\calH(\krnl_\delta\ast\mu_i)$ is $\delta$-bounded which yields the claim.
\item $\calH(\mu_i)-\calH(\krnl_\delta\ast\mu_i)\le \calH(\mu_i)<\infty$ by nonnegativity and part (a).
\end{enumerate}
\item This is obvious thanks to smoothness and concavity of the $\mob_j$; take $L=\max\limits_j \mob_j'(S^\ell_j)$.
\item This follows by means of assumption (F) on $f$ and Taylor's theorem, $\eta\in C^\infty_c(\R;\R^n)$ and the fact that $\measm\subset L^\infty(\R;\R^n)$. Note that in both cases, $\Xaux\subset L^2(\R;\R^n)$ holds.
\item Thanks to convexity and nonnegativity of $f$, this is clear.
\end{enumerate}
\end{proof}

\subsection{Time-discrete solution}\label{subsec:timediscrete}
We construct a time-discrete solution by means of the minimizing movement scheme (cf. Section \ref{sec:pre}) and introduce the \emph{Yosida penalized} entropy $\ent$, i.e.
\begin{align*}
\ent_{\tau}:\,\measm\times\measm\to\R\cup\{\infty\},\quad\Yent{\tau}{\mu}{\tilde\mu}:=\frac1{2\tau}\W_\M(\mu,\tilde\mu)^2+\ent(\mu),
\end{align*}
where $\tau\in (0,\bar\tau]$ is a given step size; and $\bar\tau>0$.

\begin{prop}[Minimizing movement (A)+(B)]\label{prop:minmov}
Let $\tau>0$ and $\tilde\mu\in \Xaux$. Then, there exists a minimizer $\mu^*\in \Xaux$ of the functional $\Yent{\tau}{\cdot}{\tilde\mu}$ on $\measm$. Moreover, one has
\begin{align}
\label{eq:addreg}
\tau \|\partial_x \mu^*\|^2_{L^2}&\le \frac2{C_f}[\calH(\tilde\mu)-\calH(\mu^*)]+C\tau,
\end{align}
where $C=C(f,\eta)>0$.
In particular, $\mu^*-\zref\in H^1(\R;\R^n)$.
\end{prop}

\begin{proof}
By Proposition \ref{prop:heatentropy}(d), $\ent$ is bounded from below. Hence, $\Yent{\tau}{\cdot}{\tilde \mu}$ is proper and bounded from below. An infimizing sequence $(\mu_k)_{k\in\N}$ in $\Xaux$,
\begin{align*}
\lim_{k\to\infty}\Yent{\tau}{\mu_k}{\tilde\mu}=\inf \Yent{\tau}{\cdot}{\tilde \mu},
\end{align*}
thus satisfies $\|\mu_k-\zref\|_{L^2}\le C$ (thanks to (F) in case (B); for case (A), this is trivial because of the uniform $L^1$ and $L^\infty$ bound on $\mu_k-\zref$) and $\W_\M(\mu_k,\tilde\mu)\le C$ for some constant $C>0$. Using Proposition \ref{prop:topo}(b) and Alaoglu's theorem yields the existence of a (non-relabelled) subsequence and a limit $\mu^*\in\Xaux$ such that $\mu_k-\zref\rightharpoonup \mu^*-\zref$ weakly in $L^2(\R;\R^n)$ and $\mu_k\stackrel{\ast}{\rightharpoonup} \mu^*$ weakly$\ast$ in $\measm$, as $k\to\infty$. Note that in case (A), finiteness of $\mom{\mu^*-\zref}$ is a consequence of the uniform bound
\begin{align*}
\mom{\mu_k-\zref}\le e^L(\mom{\tilde\mu-\zref}+C^2)<\infty,
\end{align*}
using Proposition \ref{prop:moment}. The lower semicontinuity properties from the Propositions \ref{prop:topo}(a) and \ref{prop:heatentropy}(e) show that $\mu^*$ is indeed a minimizer of $\Yent{\tau}{\cdot}{\tilde\mu}$.

In order to obtain \eqref{eq:addreg}, recall that the heat entropy $\calH$ is $0$-geodesically convex w.r.t. $\W_\M$, thanks to Proposition \ref{prop:heat_conv}. Application of the flow interchange lemma \ref{thm:flowinterchange} yields
\begin{align}
\label{eq:heatflowint}
\tau\dff^{\calH}\ent(\mu^*)&\le\calH(\tilde\mu)-\calH(\mu^*).
\end{align}
For the dissipation, we obtain (write $\mu_s:=\flow_s^{\calH}(\mu^*)$ for brevity) for small $s>0$:
\begin{align*}
-\frac{\dd}{\dd s}\ent(\mu_s)&=-\int_\R(\grd f(\mu_s)-\grd f(\zref)+\eta)^\tT\partial_{xx}\mu_s\dd x=\int_\R (\partial_x\mu_s^\tT \hess f(\mu_s)\partial_x \mu_s+\partial_x\eta^\tT\partial_x \mu_s)\dd x\\
&\ge \int_{\R} \left[C_f|\partial_x\mu_s|^2-\frac1{2C_f}|\partial_x\eta|^2-\frac{C_f}{2}|\partial_x\mu_s|^2\right]\dd x=\frac{C_f}{2}\|\partial_x\mu_s\|_{L^2}^2-\tilde C,
\end{align*}
where we used (F), the Cauchy-Schwarz and the Young inequality. Note that since $\eta\in C^\infty_c(\R;\R^n)$, $\tilde C=\tilde C(f,\eta)$ is finite. Passing to $s\searrow 0$ yields thanks to lower semicontinuity of the right-hand side
\begin{align*}
\dff^{\calH}\ent(\mu^*)&\ge \frac{C_f}{2}\|\partial_x\mu^*\|_{L^2}^2-\tilde C,
\end{align*}
from which \eqref{eq:addreg} follows by insertion into \eqref{eq:heatflowint}.
\end{proof}

The scheme \eqref{eq:minmov_gen} is well-posed and produces a sequence $(\mu_\tau^k)_{k\in\N}$ for each initial datum $\mu_\tau^0=\mu^0\in\Xaux$.
We define the \emph{time-discrete solution} $\mu_\tau:\,[0,\infty)\to\Xaux$ by piecewise constant interpolation as in \eqref{eq:disc_sol}.

The following statements are an immediate consequence of the minimizing movement:

\begin{prop}[Classical estimates (A)+(B)]\label{prop:class}
The following statements hold:
\begin{enumerate}[(a)]
\item For all $k\in\N$, one has $\ent(\mu_\tau^k)\le \ent(\mu^0)<\infty$.
\item $\displaystyle{\sum_{k=1}^\infty}\W_\M^2(\mu_\tau^k,\mu_\tau^{k-1})\le 2\tau (\ent(\mu^0)-\inf\ent).$
\item For all $T>0$ and all $s,t\in [0,T]$, one has
\begin{align*}
\W_\M(\mu_\tau(s),\mu_\tau(t))&\le \left[2(\ent(\mu^0)-\inf\ent)\max(\tau,|t-s|)\right]^{1/2}.
\end{align*}
\end{enumerate}
\end{prop}

\begin{proof}
This is classical, see for instance \cite[Ch. 3]{savare2008}.
\end{proof}

For clarity, we introduce the following notation for a given function $\varphi:[0,\infty)\to\R$: For each $\tau>0$ and $s\ge 0$, let
\begin{align*}
\varphi_\tau(s)&:=\varphi\left(\gau{\frac{s}{\tau}}\tau\right).
\end{align*}

\begin{lemma}[Discrete weak formulation (A)+(B)]\label{lemma:dweak}
Let $\alpha>0$, $\rho\in C^\infty_c(\R;\R^n)$, $\psi\in C^\infty_c((0,\infty))\cap C^0([0,\infty))$ and set $\lambda=\lambda(\alpha)=-C\left(\frac1{\alpha}+1\right)$ with $C$ from Proposition \ref{prop:convpot_dec}(b). Then, the discrete solution $\mu_\tau$ obtained from the scheme \eqref{eq:minmov_gen} satisfies the following \emph{discrete weak formulation}:
\begin{align}\label{eq:dweak}
\begin{split}
&\Bigg|\int_0^\infty \int_\R \left[\rho^\tT\mu_\tau\frac{\psi_\tau(t)-\psi_\tau(t+\tau)}{\tau}+\psi_\tau(t)[\partial_x\rho^\tT \M(\mu_\tau)\hess f(\mu_\tau)\partial_x \mu_\tau+\partial_x\rho^\tT\M(\mu_\tau)\partial_x\eta]\right]\dd x\dd t\Bigg|\\
&\le\Bigg|\alpha\int_0^\infty\int_\R \left[h_\zref(\mu_\tau)\frac{|\psi|_\tau(t)-|\psi|_\tau(t+\tau)}{\tau}+|\psi|_\tau(t)[\partial_x\mu_\tau^\tT\hess f(\mu_\tau)\partial_x \mu_\tau+\partial_x\eta^\tT\partial_x\mu_\tau]\right]\dd x\dd t\Bigg.\\ &\qquad \Bigg.+2\lambda\tau \|\psi\|_{C^0}[\ent(\mu^0)-\inf\ent]\Bigg| .
\end{split}
\end{align}
\end{lemma}

\begin{proof}
Recall that for this choice of $\lambda$, the regularized potential energy $\pot$ defined in \eqref{eq:funct_space} is $\lambda$-geodesically convex w.r.t. $\W_\M$ (cf. Proposition \ref{prop:convpot_dec}). Hence, we are in position to apply the flow interchange lemma \ref{thm:flowinterchange} to obtain for all $k\in\N$:
\begin{align}\label{eq:potflowint}
\pot(\mu_\tau^k)+\tau \dff^\pot\ent(\mu_\tau^k)+\frac{\lambda}{2}\W_\M^2(\mu_\tau^k,\mu_\tau^{k-1})&\le \pot(\mu_\tau^{k-1}).
\end{align}
For the dissipation, one has (write $\mu_s:=\flow^\pot_s(\mu_\tau^k)$ for brevity) for small $s>0$
\begin{align*}
-\frac{\dd}{\dd s}\ent(\mu_s)&=-\int_\R[\grd f(\mu_s)-\grd f(\zref)+\eta]^\tT[\alpha\partial_{xx}\mu_s+\partial_x (\M(\mu_s)\partial_x\rho)]\dd x\\
&=\alpha\int_\R \left[\partial_x \mu_s^\tT\hess f(\mu_s)\partial_x\mu_s+\partial_x\eta^\tT \partial_x \mu_s\right]\dd x+\int_\R\left[\partial_x\rho^\tT\M(\mu_s)\hess f(\mu_s)\partial_x\mu_s+\partial_x\rho^\tT \M(\mu_s)\partial_x\eta\right]\dd x,
\end{align*}
and consequently, passing to $s\searrow 0$:
\begin{align*}
\dff^\pot\ent(\mu_\tau^k)&\ge \alpha\int_\R \left[\partial_x {\mu_\tau^k}^\tT\hess f(\mu_\tau^k)\partial_x\mu_\tau^k+\partial_x\eta^\tT \partial_x \mu_\tau^k\right]\dd x+\int_\R\left[\partial_x\rho^\tT\M(\mu_\tau^k)\hess f(\mu_\tau^k)\partial_x\mu_\tau^k+\partial_x\rho^\tT \M(\mu_\tau^k)\partial_x\eta\right]\dd x.
\end{align*}
Inserting this into \eqref{eq:potflowint} and repeating this calculation with $-\rho$ in place of $\rho$ yields
\begin{align}
\label{eq:insert}
\begin{split}
&\left|\int_\R\left[\rho^\tT[\mu_\tau^k-\mu_\tau^{k-1}]+\tau \partial_x\rho^\tT\M(\mu_\tau^k)\hess f(\mu_\tau^k)\partial_x\mu_\tau^k+\tau\partial_x\rho^\tT \M(\mu_\tau^k)\partial_x\eta\right]\dd x\right|\\
&\le\left|\alpha \int_\R\left[h_\zref(\mu_\tau^k)-h_\zref(\mu_\tau^{k-1})+\tau\partial_x {\mu_\tau^k}^\tT\hess f(\mu_\tau^k)\partial_x\mu_\tau^k+\tau\partial_x\eta^\tT \partial_x \mu_\tau^k \right]\dd x+\frac{\lambda}{2}\W_\M^2(\mu_\tau^k,\mu_\tau^{k-1})\right|.
\end{split}
\end{align}
Let $\psi\in C^0([0,\infty))$ be \emph{nonnegative} and have compact support in $(0,\infty)$. We multiply the chain of inequalities \eqref{eq:insert} with $\psi((k-1)\tau)$ and take the sum over all $k\in \N$, recalling Proposition \ref{prop:class}(b) and observing
\begin{align*}
\sum_{k\in\N} \psi((k-1)\tau)[g(\mu_\tau^k)-g(\mu_\tau^{k-1})]&=\sum_{k\in\N}g(\mu_\tau^k)[\psi((k-1)\tau)-\psi(k\tau)],
\end{align*}
for an arbitrary map $g:\R^n\to\R^d$, $d\in\N$.
The resulting chain of inequalities can be expressed with the discrete solution $\mu_\tau$ as follows:
\begin{align}\label{eq:insert2}
\begin{split}
&\Bigg|\int_0^\infty \int_\R \left[\rho^\tT\mu_\tau\frac{\psi_\tau(t)-\psi_\tau(t+\tau)}{\tau}+\psi_\tau(t)[\partial_x\rho^\tT \M(\mu_\tau)\hess f(\mu_\tau)\partial_x \mu_\tau+\partial_x\rho^\tT\M(\mu_\tau)\partial_x\eta]\right]\dd x\dd t\Bigg|\\
&\le\Bigg|\alpha\int_0^\infty\int_\R \left[h_\zref(\mu_\tau)\frac{\psi_\tau(t)-\psi_\tau(t+\tau)}{\tau}+\psi_\tau(t)[\partial_x\mu_\tau^\tT\hess f(\mu_\tau)\partial_x \mu_\tau+\partial_x\eta^\tT\partial_x\mu_\tau]\right]\dd x\dd t\Bigg.\\ &\Bigg.\qquad +\lambda\tau \|\psi\|_{C^0}[\ent(\mu^0)-\inf\ent]\Bigg|.
\end{split}
\end{align}
For general $\psi\in C^\infty_c((0,\infty))\cap C^0([0,\infty))$, decompose $\psi$ into its positive and negative part and subtract the respective inequalities \eqref{eq:insert2} to obtain \eqref{eq:dweak}.
\end{proof}

\subsection{Passage to continuous time}\label{subsec:conttime}
\begin{prop}[\emph{A priori} estimates (A)]\label{prop:apriori_A}
For given $T>0$, there exist constants $C_i=C_i(T)>0$ such that for all $\tau\in (0,\bar\tau]$, the following holds:
\begin{enumerate}[(a)]
\item $\W_\M(\mu_\tau(t),\mu^0)\le C_1$ for all $t\in[0,T]$.
\item $\|\mu_\tau-\zref\|_{L^\infty([0,T];L^2)}\le C_2$.
\item $\mom{\mu_\tau(t)-\zref}\le C_3$ for all $t\in [0,T]$.
\item $\|\mu_\tau-\zref\|_{L^2([0,T];H^1)}\le C_4$.
\end{enumerate}
\end{prop}

\begin{proof}
\begin{enumerate}[(a)]
\item Using Proposition \ref{prop:class}(c) yields
\begin{align*}
\W_\M(\mu_\tau(t),\mu^0)\le [2(\ent(\mu^0)-\inf\ent)\max(\tau,t)]^{1/2}\le C_1
\end{align*}
for $0\le t\le T$ and $ 0<\tau\le \bar\tau$.
\item This is obvious thanks to the uniform bounds on $\mu_\tau-\zref$ in $L^1(\R;\R^n)$ and $L^\infty(\R;\R^n)$, respectively.
\item By part (a) and Proposition \ref{prop:moment}, one has
\begin{align*}
\mom{\mu_\tau(t)-\zref}\le e^L(\mom{\mu^0-\zref}+C_1^2)\quad\text{for all }t\in[0,T].
\end{align*}
\item In view of (b), it remains to prove that $\|\partial_x\mu_\tau\|_{L^2([0,T]; L^2)}$ is $\tau$-uniformly bounded. Set $K:=\gau{\frac{T}{\tau}}+1\le \frac{T+\bar\tau}{\tau}$ to obtain
\begin{align}
\label{eq:aprioriH1}
\int_0^T\|\partial_x\mu_\tau(t)\|_{L^2}^2\dd t&\le \sum_{k=1}^K\tau\|\partial_x \mu_\tau^k\|_{L^2}^2\le \sum_{k=1}^K\left[\frac2{C_f}(\calH(\mu_\tau^{k-1})-\calH(\mu_\tau^k))+C\tau\right],
\end{align}
where we used \eqref{eq:addreg} in the last step. In the proof of Proposition \ref{prop:heatentropy}(a), we have seen that there exist constants $\tilde C_0$, $\tilde C_1>0$ such that for all $\mu\in\Xaux$
\begin{align*}
|\calH(\mu)|&\le \tilde C_0+\tilde C_1(m+\mom{\mu-\zref}).
\end{align*}
Using (c) with $T+\bar \tau$ in place of $T$, we eventually end up with
\begin{align*}
\int_0^T\|\partial_x\mu_\tau(t)\|_{L^2}^2\dd t&\le C(T+\bar\tau)+\frac2{C_f}\left(\calH(\mu^0)+\tilde C_0+\tilde C_1(m+C_3)\right).\qedhere
\end{align*}
\end{enumerate}
\end{proof}

\begin{prop}[\emph{A priori} estimates (B)]\label{prop:apriori_B}
For given $T>0$, there exist constants $C_i=C_i(T)>0$ such that for all $\tau\in (0,\bar\tau]$, the following holds:
\begin{enumerate}[(a)]
\item $\W_\M(\mu_\tau(t),\mu^0)\le C_1$ for all $t\in[0,T]$.
\item $\|\mu_\tau-\zref\|_{L^\infty([0,T];L^2)}\le C_2$.
\addtocounter{enumi}{1}
\item $\|\mu_\tau-\zref\|_{L^2([0,T];H^1)}\le C_4$.
\end{enumerate}
\end{prop}

\begin{proof}
Part (a) is the same as for Proposition \ref{prop:apriori_A}. For part (b), thanks to Proposition \ref{prop:class}(a), for all $t>0$, one has $\ent(\mu_\tau(t))=\ent(\mu_\tau^k)\le \ent(\mu^0)$, with $k=\gau{\frac{t}{\tau}}+1$. Using Proposition \ref{prop:heatentropy}(d) yields 
\begin{align*}
\|\mu_\tau(t)-\zref\|_{L^2}\le C_2\quad\text{for all }t>0.
\end{align*}
For (d), we again proceed as before to arrive at \eqref{eq:aprioriH1}. From there, the claim obviously follows by nonnegativity of $\calH$.
\end{proof}

We now are in position to pass to the limit $\tau\searrow 0$.

\begin{prop}[Continuous-time limit (A)+(B)]\label{prop:ctl}
Let $T>0$ be given, $(\tau_k)_{k\in\N}$ be a vanishing sequence of step sizes, i.e. $\tau_k\searrow 0$ as $k\to\infty$, and $(\mu_{\tau_k})_{k\in\N}$ be the corresponding sequence of discrete solutions obtained by the minimizing movement scheme. Then, there exists a (non-relabelled) subsequence and a limit curve $\mu:\,[0,T]\to\Xaux$ such that as $k\to\infty$:
\begin{enumerate}[(a)]
\item For fixed $t\in[0,T]$, $\mu_{\tau_k}(t)\stackrel{\ast}{\rightharpoonup} \mu(t)$ weakly$\ast$ in $\measm$,
\item $\mu_{\tau_k}-\zref\rightharpoonup \mu-\zref$ weakly in $L^2([0,T];H^1(\R;\R^n))$,
\item $\mu_{\tau_k}-\zref\rightarrow \mu-\zref$ strongly in $L^2([0,T];L^2_{\mathrm{loc}}(\R;\R^n))$,
\end{enumerate}
with the properties
\begin{align}
\mu&\in C^{1/2}([0,T];(\measm,\W_\M)),\label{eq:lim1}\\
\mu-\zref&\in L^\infty([0,T];L^2(\R;\R^n))\cap L^2([0,T];H^1(\R;\R^n)).\label{eq:lim2}
\end{align}
Moreover, the limit $\mu$ is a weak solution to \eqref{eq:pdesystem} in the following sense: For all $\rho\in C^\infty_c(\R;\R^n)$ and all $\psi\in C^\infty_c((0,\infty))\cap C^0([0,\infty))$, one has
\begin{align}\label{eq:cweak}
\begin{split}
\int_0^\infty \int_\R \left[-\partial_t\psi\rho^\tT\mu+\psi[\partial_x\rho^\tT \M(\mu)\hess f(\mu)\partial_x \mu+\partial_x\rho^\tT\M(\mu)\partial_x\eta]\right]\dd x\dd t&=0.
\end{split}
\end{align}
\end{prop}

\begin{proof}
We divide the proof into several steps.

\noindent \underline{Step 1:} Weak convergence and limit properties.\\
Using the \emph{a priori} estimates in Proposition \ref{prop:apriori_A}/\ref{prop:apriori_B}(a)\&(b) together with Proposition \ref{prop:topo} and Alaoglu's theorem, we deduce the weak convergences (a)\&(b) and also the properties of the limit. Note that in case (A), finiteness of $\mom{\mu(t)-\zref}$ is a consequence of the uniform estimate from Proposition \ref{prop:apriori_A}(c). In both cases, $1/2$-H\"older continuity w.r.t. $\W_\M$ can be obtained thanks to Proposition \ref{prop:class}(c) via a refined version of the Arzelà-Ascoli theorem (cf. \cite[Thm. 3.3.1]{savare2008}).

\noindent \underline{Step 2:} Strong convergence.\\
In order to prove the strong convergence (c), we fix a bounded interval $I\subset\R$ and apply Theorem \ref{thm:ex_aub} with the admissible choices
\begin{align*}
\ban&:=\{u\in\mathscr{M}(I;S):\,u-\zref\in L^2(I;\R^n)\},\text{ endowed with }\|u\|_\ban:=\|u-\zref\|_{L^2(I)},
\end{align*}
which is isometric to a closed subset of $L^2(I;\R^n)$,
\begin{align*}
\mathcal{A}(u):=\begin{cases}\|u-\zref\|_{H^1(I)}^2,&\text{if }u-\zref\in H^1(I;\R^n)\\ +\infty,&\text{otherwise},\end{cases}
\end{align*}
which has relatively compact sublevels in $\ban$ due to the Rellich-Kondrachov compactness theorem, and
\begin{align*}
\W(u,\tilde u):=\begin{cases}\W_\M(u_\zref,(\tilde u)_\zref),&\text{if }u,\tilde u\in\dom(\mathcal{A}),\\+\infty,&\text{otherwise},\end{cases}
\end{align*}
where $u_\zref$ is to be understood as the extension of $u\in \mathscr{M}(I;S)$ to a function in $\measm$ by setting $u_\zref\equiv \zref$ outside $I$.
We verify the hypotheses \eqref{eq:hypo1}\&\eqref{eq:hypo2} for the sequence $U:=(\mu_{\tau_k})_{k\in\N}$, where -- for the sake of clarity -- we identify $\mu_{\tau_k}$ with its spatial restriction $\mu_{\tau_k}|_{[0,\infty)\times I}$ :\\ \eqref{eq:hypo1} is immediate because of the \emph{a priori} estimate from Proposition \ref{prop:apriori_A}/\ref{prop:apriori_B}(d). For \eqref{eq:hypo2}, we claim
\begin{align}
\label{eq:hypo2ver}
\sup_{k\in\N}\int_0^{T-h}\W(\mu_{\tau_k}(t+h),\mu_{\tau_k}(t))\dd t&\le \max(1,\sqrt{T+\bar\tau})\sqrt{2(\ent(\mu^0)-\inf\ent)(T+\bar \tau)h}\qquad\text{for all }h\in(0,\bar\tau), 
\end{align}
from which \eqref{eq:hypo2} follows. Indeed, for fixed $k\in\N$ and $h\in (0,\tau_k]$, one has
\begin{align*}
\int_0^{T-h}\W(\mu_{\tau_k}(t+h),\mu_{\tau_k}(t))\dd t&=\sum_{i=1}^{\gau{\frac{T}{\tau_k}}}h\W_\M(\mu_{\tau_k}^i,\mu_{\tau_k}^{i+1})\le \sqrt{2(\ent(\mu^0)-\inf\ent)}\sqrt{h^2\gau{\frac{T}{\tau_k}}}\\
&\le \sqrt{2(\ent(\mu^0)-\inf\ent)(T+\bar\tau)h},
\end{align*}
thanks to H\"older's inequality and Proposition \ref{prop:class}(b). On the other hand, for $h\in(\tau_k,\bar\tau]$, we directly get from Proposition \ref{prop:class}(c):
\begin{align*}
\int_0^{T-h}\W(\mu_{\tau_k}(t+h),\mu_{\tau_k}(t))\dd t&\le (T-h)\sqrt{2(\ent(\mu^0)-\inf\ent)h}\le (T+\bar\tau)\sqrt{2(\ent(\mu^0)-\inf\ent)h}.
\end{align*}
Hence, \eqref{eq:hypo2ver} holds and the application of Theorem \ref{thm:ex_aub} yields the existence of a (non-relabelled) subsequence which converges to (the spatial restriction to $I$ of) $\mu$ in measure w.r.t. $t\in (0,T)$. By the uniform estimate in Proposition \ref{prop:apriori_A}/\ref{prop:apriori_B}(b) and the dominated convergence theorem, we conclude that
\begin{align*}
\mu_{\tau_k}-\zref\rightarrow \mu-\zref \text{ strongly in }L^2([0,T]\times I;\R^n),
\end{align*}
proving claim (c) for a prescribed interval $I$. By a diagonal argument, setting $I_R:=[-R,R]$ and letting $R\nearrow \infty$, we deduce that (c) is true simultaneously for every bounded interval $I$, extracting a further subsequence. Moreover, we may assume that $\mu_{\tau_k}$ converges to $\mu$ pointwise almost everywhere in $[0,T]\times \R$.

\noindent\underline{Step 3:} Weak formulation.\\
Let $\rho\in C^\infty_c(\R;\R^n)$ and $\psi\in C^\infty_c((0,\infty))\cap C^0([0,\infty))$ be given and set $\alpha_k:=\sqrt{\tau_k}$ for $k\in\N$. By Lemma \ref{lemma:dweak}, $\mu_{\tau_k}$ satisfies the discrete weak formulation \eqref{eq:dweak} for each $k$, putting $\lambda_k=\lambda(\alpha_k)$ according to Lemma \ref{lemma:dweak}. Note that with this choice of $\alpha_k$, one has $\lim\limits_{k\to\infty}\lambda_k\tau_k=0$.

We first prove that
\begin{align}
\label{eq:left}
\int_0^\infty\int_\R \left[h_\zref(\mu_{\tau_k})\frac{|\psi|_{\tau_k}(t)-|\psi|_{\tau_k}(t+\tau_k)}{{\tau_k}}+|\psi|_{\tau_k}(t)[\partial_x\mu_{\tau_k}^\tT\hess f(\mu_{\tau_k})\partial_x \mu_{\tau_k}+\partial_x\eta^\tT\partial_x\mu_{\tau_k}]\right]\dd x\dd t
\end{align}
is bounded w.r.t. $k\in\N$. For the first part, since $\psi\in C^\infty_c((0,\infty))$, there exists $T'>0$ such that
\begin{align*}
\left|\int_0^\infty\int_\R h_\zref(\mu_{\tau_k})\frac{|\psi|_{\tau_k}(t)-|\psi|_{\tau_k}(t+\tau_k)}{{\tau_k}}\dd x\dd t\right|
&\le C\int_0^{T'}\int_\R |h_\zref(\mu_{\tau_k})|\dd x\dd t.
\end{align*}
In case (A), we obtain
\begin{align*}
\int_0^{T'}\int_\R |h_\zref(\mu_{\tau_k})|\dd x\dd t&\le \int_0^{T'} \left[\tilde C_0+\tilde C_1(m+\mom{\mu_{\tau_k}(t)-\zref})\right]\dd t,
\end{align*}
which is bounded thanks to Proposition \ref{prop:apriori_A}(c). In case (B), we have
\begin{align*}
\int_0^{T'}\int_\R |h_\zref(\mu_{\tau_k})|\dd x\dd t&\le \tilde C\int_0^{T'} \|\mu_{\tau_k}(t)-\zref\|_{L^2}^2\dd t,
\end{align*}
so Proposition \ref{prop:apriori_B}(b) yields boundedness. For the second part in \eqref{eq:left}, we use $\measm\subset L^\infty(\R;\R^n)$ and the Cauchy-Schwarz inequality to obtain
\begin{align*}
\left|\int_0^{\infty}\int_\R |\psi|_{\tau_k}(t)[\partial_x\mu_{\tau_k}^\tT\hess f(\mu_{\tau_k})\partial_x \mu_{\tau_k}+\partial_x\eta^\tT\partial_x\mu_{\tau_k}]\dd x\dd t\right|
&\le \overline{C}\int_0^{T'}\int_\R \left[\overline{C}'|\partial_x\mu_{\tau_k}|^2+\frac12 |\partial_x \eta|^2\right]\dd x\dd t.
\end{align*}
Thanks to $\eta\in C^\infty_c(\R;\R^n)$ and Proposition \ref{prop:apriori_A}/\ref{prop:apriori_B}(d), this is bounded.

From the dominated convergence theorem, since $\mu_{\tau_k}$ converges to $\mu$ pointwise a.e., it follows -- using $\measm\subset L^\infty(\R;\R^n)$ again -- that 
\begin{align}
\label{eq:middle1}
\begin{split}
\lim_{k\to\infty}&\left(\int_0^\infty \int_\R \left[\rho^\tT\mu_{\tau_k}\frac{\psi_{\tau_k}(t)-\psi_{\tau_k}(t+\tau_k)}{{\tau_k}}+\psi_{\tau_k}(t)\partial_x\rho^\tT\M(\mu_{\tau_k})\partial_x\eta\right]\dd x\dd t\right)\\&=\int_0^\infty \int_\R \left[-\partial_t\psi\rho^\tT\mu+\psi\partial_x\rho^\tT\M(\mu)\partial_x\eta\right]\dd x\dd t.
\end{split}
\end{align}

We now prove
\begin{align}
\label{eq:middle2}
\begin{split}
\lim_{k\to\infty}&\left( \int_0^\infty \int_\R \psi_{\tau_k}(t)\partial_x\rho^\tT \M(\mu_{\tau_k})\hess f(\mu_{\tau_k})\partial_x \mu_{\tau_k}\dd x\dd t\right)=\int_0^\infty \int_\R \psi\partial_x\rho^\tT \M(\mu)\hess f(\mu)\partial_x \mu\dd x\dd t.
\end{split}
\end{align}

First, we show that
\begin{align}
\label{eq:middle2a}
\lim_{k\to\infty}\int_0^\infty \int_\R\left(\psi_{\tau_k}(t)\partial_x\rho^\tT\M(\mu_{\tau_k})\hess f(\mu_{\tau_k})-\psi\partial_x\rho^\tT \M(\mu)\hess f(\mu)\right)\partial_x \mu_{\tau_k}\dd x\dd t&=0.
\end{align}

Using H\"older's inequality, the fact that $\psi$ and $\rho$ have compact support and Proposition \ref{prop:apriori_A}/\ref{prop:apriori_B}(d) reduces the problem to verifying
\begin{align*}
\lim_{k\to\infty}\int_0^\infty \int_\R\left|\psi_{\tau_k}(t)\partial_x\rho^\tT\M(\mu_{\tau_k})\hess f(\mu_{\tau_k})-\psi\partial_x\rho^\tT \M(\mu)\hess f(\mu)\right|^2\dd x\dd t&=0.
\end{align*}

We proceed using dominated convergence since the integrand converges pointwise a.e. to zero and the following pointwise estimate holds:
\begin{align*}
\left|\psi_{\tau_k}(t)\partial_x\rho^\tT\M(\mu_{\tau_k})\hess f(\mu_{\tau_k})-\psi\partial_x\rho^\tT \M(\mu)\hess f(\mu)\right|^2&\le C\eins{\supp \psi}\eins{\supp\rho}.
\end{align*}
The r.h.s. obviously is integrable on $(0,\infty)\times \R$. 

Second,
\begin{align}
\label{eq:middle2b}
\lim_{k\to\infty}\int_0^\infty\int_\R\psi\partial_x\rho^\tT \M(\mu)\hess f(\mu)(\partial_x \mu_{\tau_k}-\partial_x\mu)\dd x \dd t&=0,
\end{align}
since $\psi\partial_x\rho^\tT \M(\mu)\hess f(\mu)$ 
is bounded and has compact support in $[0,T']\times\R$ for some $T'>0$ and hence is an element of $L^2([0,T'];L^2(\R;\R^n))$, yielding the claim together with $\partial_x\mu_{\tau_k}\rightharpoonup \partial_x\mu$ weakly in $L^2([0,T'];L^2(\R;\R^n))$ by part (b) of this proposition. We have thus proved \eqref{eq:middle2}. Putting \eqref{eq:left}--\eqref{eq:middle2} together yields \eqref{eq:cweak}.
\end{proof}

We summarize the results of this section in the following
\begin{thm}[Existence of weak solutions]
\label{thm:exist}
Consider the initial-value problem for the system of degenerate diffusion equations with drift
\begin{align}
\partial_t \mu&=\partial_x(\M(\mu)\hess f(\mu)\partial_x\mu+\M(\mu)\partial_x\eta),\qquad t>0,\,x\in\R,\label{eq:pde1}\\
\mu(0,x)&=\mu^0(x),\qquad x\in\R,\label{eq:pde2}
\end{align}
where the mobility $\M$ is fully decoupled on the state space $S=[S^\ell,S^r]\subset \R^n$ and of the form $\M(z)=(\hess h(z))^{-1}\in\Matn$ with $h:S\to\R$ satisfying (H0)--(H3). Assume that $f:S\to\R$ satisfies (F) and $\eta\in C^\infty_c(\R;\R^n)$. \\
Suppose that $\mu^0\in\measm$ and either \\
(A) $\mu^0-\zref\in L^1(\R;\R^n)$ and $\mom{\mu^0-\zref}<\infty$ for $\zref:=S^\ell$\\
or\\
(B) $\mu^0-\zref\in L^2(\R;\R^n)$ for some $\zref\in\inn{S}$.\\
Then, there exists a function $\mu:[0,\infty)\times\R\to S$ with
\begin{align*}
\mu&\in C^{1/2}([0,T];(\measm,\W_\M)),\\
\mu-\zref&\in L^\infty([0,T];L^2(\R;\R^n))\cap L^2([0,T];H^1(\R;\R^n))
\end{align*}
for all $T>0$ satisfying \eqref{eq:pde1} in the sense of distributions and attaining the initial condition \eqref{eq:pde2}.
Additionally, in case (A), the following holds for all $t\in [0,T]$:
\begin{align*}
\|\mu(t)-\zref\|_{L^1}=\|\mu^0-\zref\|_{L^1},\quad\text{and}\quad\mom{\mu(t)-\zref}<\infty.
\end{align*}
\end{thm}


\appendix
\section{Proof of Proposition \ref{prop:convpot_dec}$(\mathrm{a})$}\label{app:ivptrans}
Since in the case at hand, the system \eqref{eq:transport} is decoupled, it suffices to prove the assertion in the scalar case $n=1$, where the mobility $\mob$ is a scalar function satisfying the properties of Section \ref{ssubsec:decpot}. Suppose that $\mu^0\in \measm$ attains values in $\inn{S}$ only, with $S=[S^\ell,S^r]\subset\R$ being an interval, and $\mu^0-\zref\in H^1(\R)$ for some $\zref\in S$. Using the transformation $u:=\mu-\zref$ and writing $\theta:=\rho_x$, we may instead consider the equation
\begin{align}
\label{eq:trnu}
\partial_tu=\partial_{xx}u+\partial_x(\mob(u+\zref)\theta),
\end{align}
together with the initial condition $u^0:=\mu^0-\zref\in H^1(\R)$ with values in $(S^\ell-\zref,S^r-\zref)\ni 0$.

Inspired from \cite[Ch. 3]{henry1981}, we write \eqref{eq:trnu} as an abstract semilinear evolution equation on $H^1(\R)$:
\begin{align}
\label{eq:abst}
\begin{split}
\dot u(t)&=-Au(t)+F(u(t)),
\end{split}
\end{align}
with $A:=-\frac{\dd^2}{\dd x^2}$, and
\begin{align*}
F(u):=\mob'(u+\zref)u_x\theta+\mob(u+\zref)\theta_x.
\end{align*}
We first prove some properties of the nonlinearity $F$.
\begin{lemma}[Properties of $F$]
\begin{enumerate}[(a)]
\item $F$ maps bounded subsets of $H^1(\R)$ onto bounded subsets of $L^2(\R)$, because for all $u\in H^1(\R)$, one has
\begin{align}
\label{eq:Fbdd}
\|F(u)\|_{L^2}&\le C_0\|u\|_{H^1}+C_1,
\end{align}
for some $C_0,C_1>0$.
\item $F$ is locally Lipschitz continuous in the following sense: If $u,v\in H^1(\R)$ with $\|u-u^0\|_{H^1}<\delta$ and $\|v-u^0\|_{H^1}<\delta$ for some $\delta>0$, then
\begin{align}
\label{eq:Flip}
\|F(u)-F(v)\|_{L^2}&\le C_2 \|u-v\|_{H^1},
\end{align}
for some $C_2=C_2(\delta,u^0)>0$.
\end{enumerate}
\end{lemma}

\begin{proof}
\begin{enumerate}[(a)]
\item By the triangle inequality, we have
\begin{align*}
\|F(u)\|_{L^2}&\le \|\mob\|_{C^1}\|\theta\|_{C^0}\|u_x\|_{L^2}+\|[\mob(u+\zref)-\mob(\zref)]\theta_x\|_{L^2}+\|\mob(\zref)\theta_x\|_{L^2}\\
&\le 2\|\mob\|_{C^1}\|\theta\|_{C^1}\|u\|_{H^1}+\|\mob(\zref)\theta_x\|_{L^2},
\end{align*}
from which the desired estimate follows since $\theta$ has compact support.
\item With $u$ and $v$ as required, one has
\begin{align*}
\|F(u)-F(v)\|_{L^2}&\le \|\theta\|_{C^1}\|\mob(u+\zref)-\mob(v+\zref)\|_{L^2}+\|\theta\mob'(u+\zref)[u_x-v_x]\|_{L^2}\\
&+\|\theta u^0_x[\mob'(u+\zref)-\mob'(v+\zref)]\|_{L^2}+\|\theta (v_x-u^0_x)[\mob'(u+\zref)-\mob'(v+\zref)]\|_{L^2}\\
&\le \|\theta\|_{C^1}\|\mob\|_{C^2}\left[\|u-v\|_{L^2}+\|u_x-v_x]|_{L^2}+(\|u^0\|_{H^1}+\|v-u^0\|_{H^1})\|u-v\|_{L^\infty}\right].
\end{align*}
Since $H^1(\R)$ is continuously embedded into $C^0(\R)$ and $\|v-u^0\|_{H^1}<\delta$, the desired estimate follows.
\end{enumerate}
\end{proof}

Let now $\delta>0$ fixed, but arbitrary and define
\begin{align*}
K_\delta:=\{u\in C^0([0,T];H^1(\R))\,|\, \|u(t)-u^0\|_{H^1}\le\delta\quad\forall t\in [0,T]\},
\end{align*}
where $T>0$ is to be determined later. $K_\delta$ is a closed subset of the Banach space $C^0([0,T];H^1(\R))$. Define a mapping $B$ on $K_\delta$ by
\begin{align*}
B(u)(t):=e^{-At}u^0+\int_0^t e^{-A(t-s)}F(u(s))\dd s\qquad\text{for }t\in[0,T].
\end{align*}
We prove the following statement: \\

\emph{There exists $T=T(\delta,u^0)>0$ sufficiently small such that
$B$ maps $K_\delta$ into itself and is a strict contraction.}\\

We first prove that $\|B(u)(t)-u^0\|_{H^1}\le \delta$ for all $t\in [0,T]$, where $T$ is sufficiently small. For all $s\in (0,t)$, one has
\begin{align*}
e^{-A(t-s)}F(u(s))=\krnl_{t-s}\ast F(u(s)),
\end{align*}
where $\krnl$ is the one-dimensional heat kernel from \eqref{eq:heatkern}. Note that for all $\sigma>0$, we have
\begin{align}
\label{eq:heatkernprop}
\|\krnl_\sigma\|_{L^1}&=A_0,\qquad \|\partial_y \krnl_\sigma\|_{L^1}=A_1\sigma^{-1/2},
\end{align}
for some constants $A_0,A_1>0$. Elementary kernel estimates yield
\begin{align}
\label{eq:exp1}
\|\krnl_t\ast u^0-u^0\|_{H^1}&\le \frac{\delta}{2},
\end{align}
for all $t\in [0,T]$, provided that $T$ is sufficiently small. For the other part, we use Young's inequality for convolutions to obtain
\begin{align*}
\left\|\int_0^t e^{-A(t-s)}F(u(s))\dd s\right\|_{H^1}&\le \int_0^t\left[\|\krnl_{t-s}\|_{L^1}+\|\partial_y \krnl_{t-s}\|_{L^1}\right]\|F(u(s))\|_{L^2}\dd s
\end{align*}
Using \eqref{eq:heatkernprop} and \eqref{eq:Fbdd}, together with the fact that $\|u(s)\|_{H^1}\le \|u^0\|_{H^1}+\delta$, since $u\in K_\delta$ yields
\begin{align}
\label{eq:exp2}
\left\|\int_0^t e^{-A(t-s)}F(u(s))\dd s\right\|_{H^1}&\le (tA_0+2\sqrt{t}A_1)(C_0\|u^0\|_{H^1}+C_0\delta+C_1)\le \frac{\delta}{2},
\end{align}
for all $t\in [0,T]$, provided that $T$ is sufficiently small. Putting \eqref{eq:exp1}\&\eqref{eq:exp2} together yields the claim. Along the same lines, it can be shown that $B(u)\in C^0([0,T];H^1(\R))$; hence $B(u)\in K_\delta$.\\

For Lipschitz-continuity, we proceed exactly as before using \eqref{eq:Flip} instead:
\begin{align*}
\|B(u)(t)-B(v)(t)\|_{H^1}&\le C_2\int_0^t (A_0+(t-s)^{-1/2}A_1)\|u(t-s)-v(t-s)\|_{H^1}\dd s\\&\le C_2(A_0T+2\sqrt{T}A_1)\|u-v\|_{C^0([0,T];H^1)},
\end{align*}
for all $t\in [0,T]$. Hence, if $T$ is sufficiently small, one has
\begin{align*}
\|B(u)-B(v)\|_{C^0([0,T];H^1)}&\le L\|u-v\|_{C^0([0,T];H^1)},
\end{align*}
for some $0\le L<1$.

Now, by Banach's fixed point theorem, $B$ possesses exactly one fixed point $u^*$ in $K_\delta$ which is, by means of \cite[Lemma 3.3.2]{henry1981}, the desired unique smooth solution to \eqref{eq:abst} on $[0,T]$. It remains to prove that $u^*(t,x)\in \inn{S}$ for all $x\in\R$ and $t\in[0,T']$, for some sufficiently small $T'>0$.

\underline{Case 1:} $\zref\in \inn{S}$. Thanks to $u^0\in H^1(\R)$, there exists $\delta_0>0$ such that 
\begin{align*}
\operatorname{dist}(u^0(x),\partial S)>\delta_0\quad \forall x\in \R.
\end{align*} 

Since $u^*\in C^0([0,T];H^1(\R))\subset C^0([0,T];C^0(\R))$, there exists $T'\in (0,T]$ such that $\|u^*(t,\cdot)-u^0\|_{C^0}<\frac{\delta_0}{2}$ for all $t\in [0,T']$. Hence, we obtain
\begin{align*}
\mathrm{dist}(u^*(t,x),\partial S)&>\frac{\delta_0}{2}\quad \forall t\in [0,T'],\,\forall x\in \R,
\end{align*}
which proves the claim.\\

\underline{Case 2:} $\zref=S^\ell$. First, as in case 1, there exists $T_1'\in (0,T]$ such that $u^*(t,x)<S^r-S^\ell$ for all $x\in\R$ and all $t\in [0,T_1']$. It remains to prove the lower bound $u^*(t,x)>0$. Let therefore $R>0$ such that $\supp(\theta)\subset [-R,R]$. Since $u^0$ is strictly positive and continuous, there exists $\delta>0$ such that $u^0(x)>\delta$ for all $x\in [-(R+1),R+1]$. Hence, we can find $T_2'\in (0,T_1']$ such that $u^*(t,x)>\frac{\delta}{2}$ for all $t\in [0,T_2']$ and all $x\in [-(R+1),R+1]$. Moreover, thanks to the smoothness of $u^*$, there exists $C_0>0$ such that $|F(u^*(s))(y)|<C_0$ for all $s\in [0,T_2']$ and all $y\in [-R,R]$.

It remains to consider the case $|x|>R+1$, $t\in [0,T_2']$, where we explicitly analyze $u^*$ by means of its fixed-point property $B(u^*)=u^*$, i.e.
\begin{align}
\label{eq:vocu}
u^*(t,x)=\int_\R \krnl_t(x-y)u^0(y)\dd y+\int_0^t\int_\R \krnl_{t-s}(x-y)F(u^*(s))(y)\dd y\dd s.
\end{align}

For the second part in formula \eqref{eq:vocu}, we immediately obtain the estimate
\begin{align*}
\left|\int_0^t\int_\R \krnl_{t-s}(x-y)F(u^*(s))(y)\dd y\dd s\right|\le C_0\int_{-R}^R\int_0^t\krnl_s(x-y)\dd s\dd y,
\end{align*}
where we recall that $|x-y|>1$. Since for fixed $v>0$, the map $g_v:\,(0,\infty)\to\R$, $g_v(s):=\frac1{\sqrt{4\pi s}}\exp\left(-\frac{v^2}{4s}\right)$ is strictly increasing for $s<\frac{v^2}{2}$, we obtain
\begin{align*}
\left|\int_0^t\int_\R \krnl_{t-s}(x-y)F(u^*(s))(y)\dd y\dd s\right|\le C_0\int_{-R}^R t\krnl_t(x-y)\dd s\dd y,
\end{align*}
if $t<\frac12$. Hence, for all $t<\min\left(T_2',\frac12,\frac{\delta}{2C_0}\right)=:T'$ and all $|x|>R+1$, formula \eqref{eq:vocu} yields
\begin{align*}
u^*(t,x)>\int_{-R}^R\left(\frac{\delta}{2}-C_0t\right)\krnl_t(x-y)\dd y,
\end{align*}
the right-hand side being nonnegative.

\underline{Case 3:} $\zref=S^r$. Here, proceed in analogy to case 2.\qed


\bibliographystyle{abbrv}
\bibliography{ref4}

\end{document}